\documentclass[11pt]{article}

\usepackage{amstext,amssymb,amsmath,amsbsy}
\usepackage{xcolor}
\usepackage{bm}
\usepackage{hyperref}
\usepackage{amscd}
\usepackage{amsfonts}
\usepackage{indentfirst}
\usepackage{verbatim}
\usepackage{amsmath}
\usepackage{amsthm}
\usepackage{enumerate}
\usepackage{graphicx}
\usepackage{color, soul}
\usepackage[OT1]{fontenc}
\usepackage[latin1]{inputenc}
\usepackage[english]{babel}
\usepackage{amssymb}
\usepackage{subfig}
\usepackage{algorithm}
\usepackage{algpseudocode}
\algrenewcommand\algorithmicrequire{\textbf{Input:}}
\algrenewcommand\algorithmicensure{\textbf{Output:}}

\newcommand{\Input}{\Require}
\newcommand{\Output}{\Ensure}

\usepackage{mathtools}

\newcommand{\R}{\mathbb{R}}

\newcommand{\ds}{\displaystyle}

\newcommand{\x}{{\bf x}}

\newcommand{\curl}{\nabla \times }

\newtheorem{Theorem}{Theorem}[section]
\newtheorem{Lemma}{Lemma}[section]

\newtheorem{Proposition}{Proposition}[section]

\newtheorem{remark}{Remark}[section]

\newtheorem*{Assumption*}{Assumption}

\newtheorem{problem}{Problem}[section]
\newtheorem*{problem*}{Problem}
\numberwithin{equation}{section}

\begin{document}

\title{Inverse initial data reconstruction for Maxwell's equations via time-dimensional reduction method}

\author{
Thuy T. Le
\thanks{
Department of Mathematics, NC State University, Raleigh, NC 27695, USA, \texttt{tle9@ncsu.edu}, corresponding author}
\and
Cong B. Van\thanks{Department of Mathematics and Statistics, University of North Carolina at
Charlotte, Charlotte, NC, 28223, USA, \texttt{cvan1@charlotte.edu}}
\and
Trong D. Dang\thanks{Faculty of Mathematics and Computer Science, University of Science, Vietnam National University, Ho Chi Minh City, Vietnam, \texttt{ddtrong@hcmus.edu.vn}} 
\and
Loc H. Nguyen\thanks{Department of Mathematics and Statistics, University of North Carolina at
Charlotte, Charlotte, NC, 28223, USA, \texttt{loc.nguyen@charlotte.edu}}  
}


\date{}
\maketitle

\begin{abstract}
We study an inverse problem for the time-dependent Maxwell system in an inhomogeneous and anisotropic medium. The objective is to recover the initial electric field $\mathbf{E}_0$ in a bounded domain $\Omega \subset \mathbb{R}^3$, using boundary measurements of the electric field and its normal derivative over a finite time interval. Informed by practical constraints, we adopt an under-determined formulation of Maxwell's equations that avoids the need for initial magnetic field data and charge density information.
To address this inverse problem, we develop a time-dimension reduction approach by projecting the electric field onto a finite-dimensional Legendre polynomial-exponential basis in time. This reformulates the original space-time problem into a sequence of spatial systems for the projection coefficients. The reconstruction is carried out using the quasi-reversibility method within a minimum-norm framework, which accommodates the inherent non-uniqueness of the under-determined setting.
We prove a convergence theorem that ensures the quasi-reversibility solution approximates the true solution as the noise and regularization parameters vanish. Numerical experiments in a fully three-dimensional setting validate the method's performance. The reconstructed initial electric field remains accurate even with $10\%$ noise in the data, demonstrating the robustness and applicability of the proposed approach to realistic inverse electromagnetic problems.
\end{abstract}

\noindent{\it Key words: 
} Time-domain Maxwell equations, inverse problem, initial condition recovery, time-dimension reduction, Legendre polynomial-exponential basis, minimum-norm solution,  convergence analysis

\noindent{\it AMS subject classification: 	
}  35R30, 35L50, 35Q61, 65M32, 78A25

\section{Introduction}

The primary objective of this paper is to address an inverse source problem involving the recovery of the initial electric field from lateral boundary measurements. Notably, our framework is economical in the sense that it does not require internal information such as the initial magnetic field or the internal charge distribution. While this relaxation introduces mathematical challenges due to the under-determined nature of the problem, it significantly reduces the burden of data acquisition and measurement in practical applications.

Let ${\bf E}$ be the electric field, governed by the well-determined Maxwell's equations
\begin{equation}
\begin{cases}
\nabla \times \left( \mu^{-1}(\mathbf{x}) \nabla \times \mathbf{E}(\mathbf{x}, t) \right) + \varepsilon(\mathbf{x}) \dfrac{\partial^2 \mathbf{E}(\mathbf{x}, t)}{\partial t^2} = \mathbf{0}, & (\mathbf{x}, t) \in \mathbb{R}^3 \times (0, \infty),
\\
\nabla \cdot (\varepsilon(\x) {\bf E}(\x, t)) = \rho(\x) &(\x, t) \in \R^3 \times (0, \infty)
\\
\mathbf{E}_t(\mathbf{x}, 0) = \varepsilon^{-1}(\x) \curl
\mathbf{H}_0(\mathbf{x}) & \mathbf{x} \in \mathbb{R}^3,
\\
\mathbf{E}(\mathbf{x}, 0) = \mathbf{E}_0(\mathbf{x}) & \mathbf{x} \in \mathbb{R}^3.
\end{cases}
\label{Maxwell}
\end{equation}
where $\mathbf{E}_0$ and $\mathbf{H}_0$ are the initial electric and magnetic fields, respectively. The $3 \times 3$ matrix-valued functions $\varepsilon(\mathbf{x})$ and $\mu(\mathbf{x})$ respectively represent the electric permittivity and magnetic permeability of the medium, and are assumed to be smooth and strictly positive definite. The scalar function $\rho(\mathbf{x})$ denotes the static charge density.
We solve the inverse problem formulated below.

\begin{problem}[Inverse initial data problem]
Assume that $\mu$ and $\varepsilon$ are known while $\rho$ and ${\bf H}_0$ are unknown.
Let ${\bf E}$ be the solution to \eqref{Maxwell}. Given the boundary data,
\begin{equation}
\mathbf{F}(\mathbf{x}, t) = \mathbf{E}(\mathbf{x}, t), \quad
\mathbf{G}(\mathbf{x}, t) = \partial_{\nu} \mathbf{E}(\mathbf{x}, t), \quad (\mathbf{x}, t) \in \partial \Omega \times (0, T),
\label{data}
\end{equation}
 reconstruct the initial electric field $\mathbf{E}_0(\mathbf{x}) = \mathbf{E}(\mathbf{x}, 0)$ within the domain $\Omega$.
\label{isp}
\end{problem}

 In general, the uniqueness of the inverse problem under consideration, namely, recovering the initial electric field $\mathbf{E}_0$ from lateral boundary data, is not guaranteed due to the under-determined nature of the formulation. This under-determined issue arises because we do not assume knowledge of the charge density $\rho$ or the initial magnetic field ${\bf H}_0$, which are standard ingredients in the classical Maxwell system. However, in the idealized case where both $\rho$ and ${\bf H}_0$ are known, and the medium is assumed to be homogeneous and isotropic with constant material parameters $\mu = \mu_0$ and $\varepsilon = \varepsilon_0$, the Maxwell system can be reformulated as a system of second-order hyperbolic equations for each component of $\mathbf{E}$.
 The principal operator for this system is $\varepsilon_0\partial_{tt} - \mu_0 \Delta.$
  In such settings, uniqueness and stability results are well-established, leveraging tools such as Carleman estimates for hyperbolic operators (see, e.g., \cite{BeilinaKlibanovBook, ClasonKlibanov:sjsc2007}).
  In contrast, this work considers the more realistic and challenging scenario in which the internal data $\rho$ and ${\bf H}_0$ are not available. To address the resulting non-uniqueness, we employ a minimum-norm formulation that selects, among all admissible solutions, the one minimizing the norm in a suitable Hilbert space, thereby compensating for the absence of $\rho$. To overcome the lack of information on ${\bf H}_0$, we apply a time-domain reduction technique. This combined strategy yields a well-posed formulation of the inverse problem, representing a key theoretical contribution of our work.

Inverse problems for Maxwell's equations have been extensively investigated in various settings, including source identification, material parameter reconstruction, and imaging applications such as magnetoencephalography and non-destructive testing. A broad overview of inverse problems in the Maxwell framework is provided in~\cite{romanov1994inverse}. Classical formulations in the frequency (time-harmonic) domain typically focus on reconstructing current sources using boundary or far-field measurements. Owing to the inherent ill-posedness of these problems, multi-frequency techniques have been widely employed to improve both stability and resolution.
Despite notable advances, comprehensive stability results for inverse source problems in the full Maxwell system under general conditions remain limited. A significant contribution in this direction is found in~\cite{bao2020stability}, where increasing stability was demonstrated for both elastic and electromagnetic wave equations with multi-frequency data. The analysis established that higher-frequency source components become negligible, thereby enhancing the quality of reconstructions.
In \cite{HarrisLeNguyen2024}, a fast and robust numerical method has been developed to reconstruct point-like or small-volume sources in the time-harmonic Maxwell equations from Cauchy data at a fixed frequency. This approach leverages imaging functions and asymptotic expansions, offering stable recovery of source locations and moments without relying on Carleman estimates.
In the context of anisotropic or inhomogeneous media,~\cite{li2005anisotropic} and~\cite{isakov2021} obtained uniqueness and increasing stability results through the use of Carleman estimates and analytic continuation techniques in the frequency domain. These studies underscore the importance of combining broadband data with structural assumptions on the medium to mitigate the effects of instability. Additionally,~\cite{albanese2006inverse} investigated non-uniqueness in determining volume current densities from boundary observations, showing that uniqueness can be restored under structural assumptions, such as dipole or surface-supported sources.
On the computational side, several reconstruction strategies have been developed. For example,~\cite{griesmaier2018windowed} introduced a qualitative imaging approach based on the windowed Fourier transform for time-harmonic Maxwell equations, enabling effective identification of radiating source supports. Similarly,~\cite{ammari2002inverse} addressed the recovery of dipole sources in magnetoencephalography using a low-frequency asymptotic expansion, establishing uniqueness, stability, and a practical inversion method.
Time-domain inverse source problems have also attracted growing interest. In~\cite{benoit2015time}, a numerical optimization framework was developed to identify time-dependent current sources, motivated by applications in fault detection and antenna design. A convergent algorithm for reconstructing temporal source functions from normal electric field measurements was proposed in~\cite{slodicka2016recovery}, while~\cite{kang2018reconstruction} introduced a potential-based formulation and proved both uniqueness and convergence for the associated reconstruction scheme.
Despite these advances, the problem of recovering the initial condition $\mathbf{E}_0(\mathbf{x})$ from finite-time boundary measurements, distinct from recovering source terms or material parameters, has received relatively little attention. This problem is particularly challenging due to the vectorial and coupled nature of Maxwell's system, which complicates inversion procedures.
In this work, we address this gap by studying the recovery of the initial electric field $\mathbf{E}_0(\mathbf{x})$ using boundary measurements of $\mathbf{E}(\mathbf{x}, t)$ and its normal derivative on $\partial \Omega \times (0, T)$. This setting is motivated by practical scenarios in which internal field data are inaccessible and only surface measurements can be obtained. 
Our objective is to design an effective numerical reconstruction method based on the theoretical structure of the problem and informed by the broader literature on inverse wave and electromagnetic problems.

For completeness, we list here some important works for the inverse initial data problem for hyperbolic equations.
The problem of recovering initial conditions in hyperbolic wave models has been extensively studied. In particular, when wave propagation occurs in free space, a number of explicit reconstruction formulas have been developed; see, for example, \cite{DoKunyansky:ip2018, Haltmeier:cma2013, Natterer:ipi2012, Linh:ipi2009}. Other commonly used approaches include the time reversal method \cite{Hristova:ip2009, HristovaKuchmentLinh:ip2006, KatsnelsonNguyen:aml2018, Stefanov:ip2009, Stefanov:ip2011}, the quasi-reversibility method \cite{ClasonKlibanov:sjsc2007, LeNguyenNguyenPowell:JOSC2021}, and a variety of iterative reconstruction techniques \cite{Huangetal:IEEE2013, Paltaufetal:ip2007, Paltaufetal:osa2002}.
These contributions primarily focus on simplified wave propagation models assuming isotropic and non-dissipative media. More advanced models incorporating damping or attenuation effects have also been analyzed in the literature; see \cite{Acosta:jde2018, Ammarietal:cm2011, Ammarielal:sp2012, Burgholzer:pspie2007, Haltmeier:jmiv2019, Homan:ipi2013, Kowar:SISI2014, Kowar:sp2012, Nachman1990}.

Our approach introduces a time-dimension reduction framework that transforms the original space-time inverse problem into a sequence of coupled spatial problems. This is achieved by projecting the time-dependent electric field onto a Legendre polynomial-exponential basis first developed in \cite{TrongElastic}, which ensures rapid decay of higher-order temporal modes and facilitates numerical stability. Each spatial mode satisfies a curl-curl-type equation involving the projection of the second time derivative of the true field. These equations are solved using a quasi-reversibility method formulated in a minimum-norm setting, which enables stable recovery under an underdetermined formulation. The method does not require knowledge of the initial magnetic field or the charge density and is designed to accommodate noisy boundary measurements. We prove convergence of the regularized solutions to the true electric field under appropriate parameter regimes and noise levels. Numerical results in fully three-dimensional configurations demonstrate the method's robustness and accuracy, even in the presence of substantial noise in the input data.

The remainder of the paper is organized as follows. In Section \ref{sec:legendre_basis}, we present the forward model for Maxwell's equations under an under-determined formulation and describe the inverse problem setup. Section \ref{sec_min} introduces the minimum-norm framework used to address the non-uniqueness arising from the under-determined setting. In Section \ref{sec4}, we develop the time-dimension reduction method based on a Legendre polynomial-exponential basis and derive the corresponding spatial systems. Section \ref{quasi} provides the convergence analysis of the regularized solution in the presence of noisy boundary data. Section \ref{sec6} presents numerical experiments that demonstrate the accuracy and stability of the proposed approach. Finally, Section \ref{sec7} concludes the paper with a summary and discussion of possible extensions.

\section{Preliminaries on the Legendre polynomial-exponential basis}
\label{sec:legendre_basis}

In this section, we review the construction and essential properties of the Legendre polynomial exponential basis, originally introduced in~\cite{TrongElastic}. This basis plays a central role in the time-dimensional reduction framework due to its favorable approximation and spectral properties.

Let $P_n(x)$ denote the classical Legendre polynomial of degree $n \geq 0$ on the interval $(-1, 1)$, defined by Rodrigues' formula:
\[
P_n(x) = \frac{1}{2^n n!} \frac{d^n}{dx^n}(x^2 - 1)^n.
\]
To define an analogous family on the interval $(0, T)$, we apply the change of variable $x = \frac{2t}{T} - 1$, yielding the rescaled orthonormal system:
\[
Q_n(t) = \sqrt{\frac{2n + 1}{T}} \, P_n\left( \frac{2t}{T} - 1 \right), \quad t \in (0, T),
\]
which forms an orthonormal basis of $L^2(0, T)$. Multiplying this by an exponential weight, we obtain the Legendre polynomial-exponential basis:
\[
\Psi_n(t) := e^t Q_n(t), \quad t \in (0, T).
\]
The collection $\{\Psi_n\}_{n \geq 0}$ constitutes an orthonormal basis of the square-integrable space with the weight $e^{-2t}$
\[
L^2_{e^{-2t}}(0, T) := \left\{ u \in L^2(0, T) \ \middle| \ \int_0^T e^{-2t} |u(t)|^2 \, dt < \infty \right\},
\]
equipped with the inner product
\[
\langle u, v \rangle_{L^2_{e^{-2t}}(0, T)} := \int_0^T e^{-2t} u(t) v(t) \, dt.
\]

Let $H$ be a Hilbert space. Throughout this paper, we denote by
\[
L^2_{e^{-2t}}((0, T); H) := \left\{ u : (0, T) \to H \ \middle| \ \int_0^T e^{-2t} \|u(t)\|_H^2 \, dt < \infty \right\}
\]
the space of $H$-valued functions that are square-integrable over $(0, T)$ with respect to the exponential weight $e^{-2t}$.

In particular, when $H$ is $H^p(\Omega)^3$ with $p \geq 0$,
we view the space  $L^2_{e^{-2t}}((0, T); H)$ as
\[
L^2_{e^{-2t}}((0, T); H) = \left\{ u : \Omega \times (0, T) \to \mathbb{R} \ \middle| \ \int_0^T e^{-2t} \|u(\cdot, t)\|_{H}^2 \, dt < \infty \right\},
\]
where the norm $\|u(\cdot, t)\|_H$ is interpreted with respect to the spatial variable for each fixed time $t$.
Also in the case when $H$ is $H^p(\Omega)^3$ with $p \geq 0$, we introduce several useful operators for functions $u \in L^2_{e^{-2t}((0, T); H)}$:

\begin{enumerate}
    \item Let $u(\mathbf{x}, t) \in L^2_{e^{-2t}}((0, T); H)$. The vector of its $(N+1)$ first Fourier coefficients with respect to the orthonormal basis $\{\Psi_n\}_{n=0}^N$ is defined as
    \begin{equation*}
        \mathbb{F}^N[u](\mathbf{x}) := 
        \begin{bmatrix}
            u_0(\mathbf{x}) & u_1(\mathbf{x}) & \dots & u_N(\mathbf{x})
        \end{bmatrix}^\top,
        \quad \x \in \Omega,
    \end{equation*}
    where each coefficient is given by
    \begin{equation}
    u_n(\mathbf{x}) := \int_0^T e^{-2t} u(\mathbf{x}, t) \Psi_n(t) \, dt, \quad n \geq 0,  \x \in \Omega.
    \label{un}
    \end{equation}
    
    \item Given a coefficient vector $\mathbf{U}(\mathbf{x}) := \begin{bmatrix}
    u_0(\mathbf{x}) & u_1(\mathbf{x}) & \dots & u_N(\mathbf{x})
    \end{bmatrix}^\top \in H^{N + 1}$, its associated space-time expansion is defined as
    \begin{equation}
        \mathbb{S}^N[\mathbf{U}](\mathbf{x}, t) := \sum_{n = 0}^N u_n(\mathbf{x}) \Psi_n(t), \quad (\x, t) \in \Omega \times (0, T).
        \label{SN1}
    \end{equation}
    
    \item Let $u(\mathbf{x}, t) \in L^2_{e^{-2t}}((0, T); H)$. Its projection onto the finite-dimensional subspace
    \[
    \mathbb{V}^N := \mathrm{span} \{\Psi_0, \Psi_1, \dots, \Psi_N\}
    \]
    is given by the truncated expansion
    \begin{equation}
        \mathbb{P}^N[u](\mathbf{x}, t) := \sum_{n = 0}^N u_n(\mathbf{x}) \Psi_n(t),
        \label{SN}
    \end{equation}
    where the coefficients $\{u_n(\mathbf{x})\}_{n = 0}^N$ are as defined in item 1.
\end{enumerate}

\begin{remark}
Let $p \geq 0$ and $H = H^p(\Omega)^3$.
For any $N \geq 0$, the composition $\mathbb{S}^N \circ \mathbb{F}^N$ coincides with the orthogonal projection $\mathbb{P}^N$ onto the subspace $\mathbb{V}^N$. That is,
\begin{equation*}
\mathbb{S}^N[\mathbb{F}^N[u]] = \mathbb{P}^N[u], \quad \text{for all } u(\mathbf{x}, t) \in L^2_{e^{-2t}}((0, T); H).
\end{equation*}

Furthermore, Parseval's identity implies the following norm relation:
\begin{equation}
\|\mathbb{P}^N[u]\|_{L^2_{e^{-2t}}((0, T); H)} = \|\mathbb{S}^N[\mathbf{U}]\|_{L^2_{e^{-2t}}((0, T); H)} = \|\mathbf{U}\|_{H^{(N+1)}} \leq \|u\|_{L^2_{e^{-2t}}((0, T); H)},
\label{2.4}
\end{equation}
where $\mathbf{U} = \mathbb{F}^N[u]$ denotes the vector of Fourier coefficients.
\end{remark}

%
%

The following proposition guarantees that the second time derivative of a sufficiently smooth function can be recovered from the second derivatives of its Fourier expansion in the weighted space.

\begin{Proposition}
Let $u \in H^k((0, T); L^2(\Omega))$ for some $k \geq 5$. Then the expansion
\begin{equation*}
\sum_{n=0}^\infty \langle u(\cdot, \cdot), \Psi_n \rangle_{L^2_{e^{-2t}}(0, T)} \Psi_n''(t)
\longrightarrow \frac{\partial^2 u}{\partial t^2}
\end{equation*}
 in $L^2_{e^{-2t}}((0, T); L^2(\Omega))$. 
\label{Prop_Trong}
\end{Proposition}

The following proposition guarantees that the derivatives of the basis functions $\Psi_n$, $n \geq 0$, are nontrivial. This non-vanishing property is essential for maintaining the accuracy of the time-dimensional reduction method.

\begin{Proposition}
For every $n \geq 0$, the function $\Psi_n(t) = e^t Q_n(t)$ is infinitely differentiable on the interval $(0, T)$, and none of its derivatives of any order vanishes identically on $(0, T)$.
\label{prop2.1}
\end{Proposition}

For details of the proofs of Propositions \ref{Prop_Trong} and \ref{prop2.1}, see~\cite{TrongElastic}.
We emphasize that the Legendre polynomial-exponential basis offers significant structural advantages over previously studied polynomial-exponential systems (e.g.,~\cite{Klibanov:jiip2017}). Its orthogonality, regularity, and spectral decay properties make it especially well-suited for the analysis and numerical implementation of time-reduced inverse methods.
For a complete account of theoretical results and proofs concerning approximation properties and convergence behavior, we refer the reader to~\cite{TrongElastic}.

The following lemma plays an essential role in establishing our convergence result. Its proof can be found in \cite[Lemma 4.2]{TrongElastic}.

\begin{Lemma}
Let $T > 0$. Then there exists a constant $C > 0$ such that
\[
\|v'\|^2_{L^2_{e^{-2t}}(0, T)} \leq C \left( \|v\|^2_{L^2_{e^{-2t}}(0, T)} + \|v''\|^2_{L^2_{e^{-2t}}(0, T)} \right)
\]
for all $v \in H^2(0, T)$.
\label{lem4.2}
\end{Lemma}

\section{The uniqueness of the minimum-norm solution }\label{sec_min}

Problem~\ref{isp} can be reformulated as the task of determining a vector-valued function $\mathbf{E}$ that satisfies the following boundary value problem:
\begin{equation}
\begin{cases}
    \nabla \times \left( \mu^{-1}(\mathbf{x}) \nabla \times \mathbf{E}(\mathbf{x}, t) \right) + \varepsilon(\mathbf{x}) \dfrac{\partial^2 \mathbf{E}(\mathbf{x}, t)}{\partial t^2} = \mathbf{0}, & (\mathbf{x}, t) \in \Omega_T, \\
    \mathbf{E}(\mathbf{x}, t) = \mathbf{F}^*(\mathbf{x}, t), & (\mathbf{x}, t) \in \Gamma_T, \\
    \partial_{\nu} \mathbf{E}(\mathbf{x}, t) = \mathbf{G}^*(\mathbf{x}, t), & (\mathbf{x}, t) \in \Gamma_T,
\end{cases}
\label{3_1}
\end{equation}
where $\mathbf{F}^*$ and $\mathbf{G}^*$ denote the exact (noise-free) boundary data defined in~\eqref{data} and $\Gamma_T = \partial\Omega \times (0,T)$. Once~\eqref{3_1} is solved, the initial condition can be recovered by setting $\mathbf{E}_0 = \mathbf{E}(\cdot, 0)$.

The existence of a solution to~\eqref{3_1} is guaranteed, as $\mathbf{E}$ can be taken as the solution to the corresponding forward problem.
However, the solution to system~\eqref{3_1} is generally not unique due to the absence of initial conditions and the lack of information about the divergence of the field. In the full Maxwell framework, the vector field $\mathbf{E}$ is typically governed not only by the curl-curl dynamics but also by an accompanying condition on the divergence of $\varepsilon\mathbf{E}$, which constrains the admissible solution space. However, in problem~\eqref{3_1}, we do not assume prior knowledge of $\nabla\cdot (\varepsilon\mathbf {E})$ anywhere in the domain. This lack of divergence information might lead to multiple fields satisfying the same boundary data and evolution equation, thereby giving rise to nonuniqueness. To resolve this ambiguity and ensure well-posedness, we compute a solution within a minimum-norm framework, selecting the field with the lowest Sobolev norm among all those satisfying the given system and boundary conditions. This approach yields a unique, stable, and physically meaningful solution, even without full information about the system.

\begin{Theorem}
Let the admissible set of solutions be defined by
\[
\mathcal{S} := \left\{ \mathbf{E} \in L^2_{e^{-2t}}((0, T); H^3(\Omega)^3) \cap H^2((0, T); L^2(\Omega)^3) \ \middle| \ \mathbf{E} \text{ satisfies } \eqref{3_1} \right\}.
\]
Suppose that $\mu$, $\epsilon$, $\mathbf{F}^*$, and $\mathbf{G}^*$ are sufficiently smooth. Then, problem~\eqref{3_1} admits a unique solution $\mathbf{E}^* \in L^2((0, T); H^3(\Omega)^3) \cap H^2((0, T); L^2(\Omega)^3)$ satisfying
\[
\| \mathbf{E}^* \|_{L^2_{e^{-2t}}((0, T); H^3(\Omega)^3)} = \min \left\{ \| \mathbf{E} \|_{L^2_{e^{-2t}}((0, T); H^3(\Omega)^3)} : \mathbf{E} \in \mathcal{S} \right\}.
\]
\label{thm31}
\end{Theorem}

\begin{remark}
The choice of the stronger Hilbert space $L^2((0, T); H^3(\Omega)^3)$ in Theorem~\ref{thm31} may appear unnecessarily restrictive compared to the more natural space $L^2((0, T); H^2(\Omega)^3)$. However, this choice is made to ensure consistency with the convergence analysis presented in Section~\ref{quasi} for the quasi-reversibility method, where the regularization term involves the $H^3$-norm.
\end{remark}

\begin{proof}[Proof of Theorem \ref{thm31}]

Since $\mu$, $\varepsilon$, ${\bf F}^*$, and ${\bf G}^*$ are sufficiently smooth, the forward problem admits a smooth solution that lies in the admissible set $\mathcal{S}$. Hence, $\mathcal{S}$ is nonempty.
Define the infimum
\[
\alpha := \inf \left\{ \| \mathbf{E} \|_{L^2_{e^{-2t}}((0, T); H^3(\Omega)^3)} : \mathbf{E} \in \mathcal{S} \right\}.
\]
Then for each $n \in \mathbb{N}$, there exists $\mathbf{E}^n \in \mathcal{S}$ such that
\begin{equation}
\alpha \leq \| \mathbf{E}^n \|_{L^2_{e^{-2t}}((0, T); H^3(\Omega)^3)} < \alpha + \frac{1}{n}.
\label{to-min}
\end{equation}
Since the sequence $\{ \mathbf{E}^n \}_{n \in \mathbb{N}}$ is bounded in $L^2_{e^{-2t}}((0, T); H^3(\Omega)^3)$, due to \eqref{3_1}, \[\partial^2_{tt}\mathbf{E}^n = -\varepsilon^{-1}(\x) \curl (\mu^{-1}(\x) \curl \mathbf{E}^n )\]  is bounded in $L^2((0, T); H^1(\Omega)^3)$.
By compact embedding theorems,
there exists a subsequence (still denoted by $\mathbf{E}^n$) such that
\begin{align*}
	\mathbf{E}^n &\rightharpoonup \mathbf{E}^* \text{ weakly in } L^2_{e^{-2t}}(0,T; (H^{3}(\Omega))^3),\\
	\partial_{tt}\mathbf{E}^n &\rightharpoonup \partial_{tt}\mathbf{E}^* \text{ weakly in } L^2(0,T; (H^1(\Omega))^3).
\end{align*}
Hence, ${\bf E}^* \in \mathcal{S}.$
By weak lower semicontinuity and \eqref{to-min}, it follows that
\[
\|\mathbf{E}\|_{L^2_{e^{-2t}}(0,T; (H^{3}(\Omega))^3)} = \alpha.
\]
To establish uniqueness, suppose there exists another minimizer $\mathbf{E}^1 \in \mathcal{S}$ such that
\[
\| \mathbf{E}^1 \|_{L^2_{e^{-2t}}((0, T); H^3(\Omega)^3)}= \alpha.
\]
The convexity of $\mathcal{S}$ implies $(\mathbf{E}^* + \mathbf{E}^1)/2 \in \mathcal{S}$. Therefore,
\begin{equation}
\left\| \frac{\mathbf{E}^* + \mathbf{E}^1}{2} \right\|_{L^2_{e^{-2t}}((0, T); H^3(\Omega)^3)} \geq \alpha.
\label{3_2}
\end{equation}
Applying the parallelogram identity and using \eqref{3_2} yields
\begin{align*}
\| \mathbf{E}^* - &\mathbf{E}^1 \|_{L^2_{e^{-2t}}((0, T); H^3(\Omega)^3)}^2  
\\
&= 2\| \mathbf{E}^* \|_{L^2_{e^{-2t}}((0, T); H^3(\Omega)^3)}^2 
 + 2\| \mathbf{E}^1 \|_{L^2_{e^{-2t}}((0, T); H^3(\Omega)^3)}^2
\\
&
 - 4\left\| \frac{\mathbf{E}^* + \mathbf{E}^1}{2} \right\|_{L^2_{e^{-2t}}((0, T); H^3(\Omega)^3)}^2  \\
&\leq 2\alpha + 2\alpha - 4\alpha = 0,
\end{align*}
which implies that $\mathbf{E}^* = \mathbf{E}^1$. This completes the proof.
\end{proof}

Throughout the remainder of the paper, $C$ denotes a generic constant that may change from one occurrence to another. Its value depends only on the domain $\Omega$, the final time $T$, and the norm $\|{\bf E}^*\|_{L^2_{e^{-2t}}((0, T); H^3(\Omega)^3)}$.

\begin{remark}
In the absence of uniqueness for the system~\eqref{3_1}, the minimum-norm framework offers a principled approach for selecting a unique solution from the potentially infinite solution set. By minimizing a suitable norm over the admissible space $\mathcal{S}$, we obtain the solution with the lowest regularity cost in the chosen Sobolev space. 
Mathematically, the well-posedness of the minimum-norm problem follows from the strict convexity and weak lower semicontinuity of the objective functional in Hilbert spaces, as well as the closed convexity of the constraint set. Physically, the minimum-norm solution may be interpreted as the least energetic or most regular field compatible with the given boundary data, aligning with principles of parsimony and stability.

\end{remark}

\begin{remark}
Aside from the minimum-norm framework, another way to resolve the non-uniqueness is to enrich the inverse problem with additional information. In particular, if the material parameters $\mu$ and $\varepsilon$ are scalar-valued and sufficiently regular, and if both the charge density $\rho$ and the initial magnetic field $\mathbf{H}_0$ are known, then the inverse problem becomes well-posed. Under these assumptions, reconstructing the initial electric field $\mathbf{E}_0$ reduces to finding a function $\mathbf{E}^*$ that satisfies the following system:
\begin{equation}
\begin{cases}
\nabla \times \left( \mu^{-1}(\mathbf{x}) \nabla \times \mathbf{E}(\mathbf{x}, t) \right) + \varepsilon(\mathbf{x}) \dfrac{\partial^2 \mathbf{E}(\mathbf{x}, t)}{\partial t^2} = \mathbf{0}, & (\mathbf{x}, t) \in \Omega_T, \\[0.5em]
\nabla \cdot \left( \varepsilon(\mathbf{x}) \mathbf{E}(\mathbf{x}, t) \right) = \rho(\mathbf{x}), & (\mathbf{x}, t) \in \Omega_T, \\[0.5em]
\partial_t \mathbf{E}(\mathbf{x}, 0) = \varepsilon^{-1}(\mathbf{x}) \nabla \times \mathbf{H}_0(\mathbf{x}), & \mathbf{x} \in \Omega, \\[0.5em]
\mathbf{E}(\mathbf{x}, t) = \mathbf{F}^*(\mathbf{x}, t), & (\mathbf{x}, t) \in \Gamma_T, \\[0.5em]
\partial_{\nu} \mathbf{E}(\mathbf{x}, t) = \mathbf{G}^*(\mathbf{x}, t), & (\mathbf{x}, t) \in \Gamma_T.
\end{cases}
\label{3_3}
\end{equation}
Again, the existence of ${\bf E}^*$ is obvious since ${\bf F}^*$ and ${\bf G}^*$ represent the perfect measured data. The uniqueness of \eqref{3_3} is based on a Carleman estimate.
In fact,
assume that system~\eqref{3_3} admits two solutions, denoted by $\mathbf{E}_1$ and $\mathbf{E}_2$. Define their difference $\mathbf{W} = \mathbf{E}_1 - \mathbf{E}_2$. Then $\mathbf{W}$ satisfies the following homogeneous problem:
\begin{equation*}
\begin{cases}
\nabla \times \left( \mu^{-1}(\mathbf{x}) \nabla \times \mathbf{W}(\mathbf{x}, t) \right) + \varepsilon(\mathbf{x}) \dfrac{\partial^2 \mathbf{W}(\mathbf{x}, t)}{\partial t^2} = \mathbf{0}, & (\mathbf{x}, t) \in \Omega_T, \\[0.5em]
\nabla \cdot \left( \varepsilon(\mathbf{x}) \mathbf{W}(\mathbf{x}, t) \right) = 0, & (\mathbf{x}, t) \in \Omega_T, \\[0.5em]
\partial_t \mathbf{W}(\mathbf{x}, 0) = \mathbf{0}, & \mathbf{x} \in \Omega, \\[0.5em]
\mathbf{W}(\mathbf{x}, t) = \mathbf{0}, & (\mathbf{x}, t) \in \Gamma_T, \\[0.5em]
\partial_{\nu} \mathbf{W}(\mathbf{x}, t) = \mathbf{0}, & (\mathbf{x}, t) \in \Gamma_T.
\end{cases}
\end{equation*}

Applying the vector calculus identity for the curl of a product, we compute
\begin{align*}
\nabla \times \left( \mu^{-1}(\mathbf{x}) \nabla \times \mathbf{W} \right) 
&= \mu^{-1}(\mathbf{x}) \nabla \times \nabla \times \mathbf{W} 
+ \nabla \mu^{-1}(\mathbf{x}) \times (\nabla \times \mathbf{W}) \\
&= \mu^{-1}(\mathbf{x}) \left[ -\Delta \mathbf{W} + \nabla (\nabla \cdot \mathbf{W}) \right]
+ \nabla \mu^{-1}(\mathbf{x}) \times (\nabla \times \mathbf{W}),
\end{align*}
for all $(\mathbf{x}, t) \in \Omega_T$. Substituting this into the Maxwell equation yields
\begin{equation}
\mu^{-1}(\mathbf{x}) \left[ -\Delta \mathbf{W} + \nabla (\nabla \cdot \mathbf{W}) \right]
+ \varepsilon(\mathbf{x}) \dfrac{\partial^2 \mathbf{W}}{\partial t^2} = \mathbf{0}, 
\quad (\mathbf{x}, t) \in \Omega_T.
\label{eq:laplace-like}
\end{equation}

Meanwhile, applying the product rule to the divergence condition gives
\[
\nabla \cdot (\varepsilon(\mathbf{x}) \mathbf{W}) = \varepsilon(\mathbf{x}) \nabla \cdot \mathbf{W} 
+ \nabla \varepsilon(\mathbf{x}) \cdot \mathbf{W} = 0,
\]
from which it follows that
\begin{equation}
\nabla \cdot \mathbf{W}(\mathbf{x}, t) 
= -\varepsilon^{-1}(\mathbf{x}) \nabla \varepsilon(\mathbf{x}) \cdot \mathbf{W}(\mathbf{x}, t).
\label{eq:divW}
\end{equation}

Substituting \eqref{eq:divW} into \eqref{eq:laplace-like}, we obtain the identity
\begin{equation*}
 -\Delta \mathbf{W} 
- \nabla \left( \varepsilon^{-1}(\mathbf{x}) \nabla \varepsilon(\mathbf{x}) \cdot \mathbf{W} \right) 
+ c(\x) \dfrac{\partial^2 \mathbf{W}}{\partial t^2} = \mathbf{0}, 
\quad (\mathbf{x}, t) \in \Omega_T.
\end{equation*}
where $c(\x) = \mu(\x) \varepsilon(\x).$
Hence, the governing equation for $\mathbf{W}$ is a hyperbolic equation whose principal part is given by the operator $c(\mathbf{x}) \, \partial_{tt} - \Delta$. The uniqueness and stability of the associated inverse source problem can, in principle, be established via a Carleman estimate, as demonstrated in~\cite{BeilinaKlibanovBook, ClasonKlibanov:sjsc2007}. One technical obstacle arises: the result in~\cite{BeilinaKlibanovBook, ClasonKlibanov:sjsc2007} addresses scalar-valued functions, while $\mathbf{W}$ is a vector field. Nonetheless, the argument therein can be extended to the vector case with only minor modifications. Since the uniqueness of system~\eqref{3_3} depends on the assumption that $\rho$ and $\mathbf{H}_0$ are known, a condition not imposed in our main framework, we present this uniqueness discussion only for illustrative purposes. The full proof is omitted, and we refer the interested reader to~\cite{BeilinaKlibanovBook, ClasonKlibanov:sjsc2007} for further details.
\end{remark}

\section{The Time Dimension Reduction Method} \label{sec4}

To reduce the computational complexity of the inverse problem, we introduce a spectral time-projection method that eliminates the time variable from the Maxwell system. This reduction transforms the original $(3+1)$-dimensional formulation into a purely three-dimensional problem, thereby facilitating the numerical solution.

For each $N \geq 1$, we apply the time-projection operator $\mathbb{P}^N$, as defined in~\eqref{SN}, to both sides of the equation~\eqref{Maxwell}. This yields the projected equation:
\begin{equation}
\sum_{n = 0}^N \nabla \times \left( \mu^{-1}(\mathbf{x}) \nabla \times \mathbf{e}_n(\mathbf{x}) \right) \Psi_n(t)
+ \varepsilon(\mathbf{x}) \, \mathbb{P}^N\left[ \frac{\partial^2 \mathbf{E}(\mathbf{x}, t)}{\partial t^2} \right]
= \mathbf{0}, \quad \mathbf{x} \in \mathbb{R}^3, \quad t > 0,
\label{3.1}
\end{equation}
where
the vector
\begin{equation*}    
\begin{bmatrix} \mathbf{e}_0 & \mathbf{e}_1 & \dots & \mathbf{e}_N \end{bmatrix}^{\top} = \mathbb{F}^N[\mathbf{E}]
\end{equation*}
represents the Fourier coefficients of $\mathbf{E}$.
To eliminate the time variable, we multiply both sides of \eqref{3.1} by $e^{-2t} \Psi_m(t)$ and integrate over $(0, T)$ for each $m \in \{0, 1, \dots, N\}$. Using the orthonormality of $\{\Psi_n\}_{n \geq 0}$ in $L^2_{e^{-2t}}(0, T)$, we obtain
\begin{equation}
\nabla \times \left( \mu^{-1}(\mathbf{x}) \nabla \times \mathbf{e}_m(\mathbf{x}) \right)
+ \varepsilon(\mathbf{x}) \int_{0}^T e^{-2t} \mathbf{E}_{tt}(\mathbf{x}, t) \Psi_m(t) \, dt
= {\bf 0}, \quad \mathbf{x} \in \Omega.
\label{3.2}
\end{equation}
Define the coefficient matrix
\begin{equation*}
s_{mn} = \langle \Psi''_n, \Psi_m\rangle_{L^2_{e^{-2t}}(0, T)} = \int_0^T e^{-2t} \Psi_n''(t) \Psi_m(t) \, dt,
\end{equation*}
and add the term $\varepsilon(\mathbf{x}) \sum_{n = 0}^N s_{mn} \mathbf{e}_n(\mathbf{x})$ to both sides of \eqref{3.2}, giving
\begin{multline}
\nabla \times \left( \mu^{-1}(\mathbf{x}) \nabla \times \mathbf{e}_m(\mathbf{x}) \right)
+ \varepsilon(\mathbf{x}) \sum_{n = 0}^N s_{mn} \mathbf{e}_n(\mathbf{x}) \\
= \varepsilon(\mathbf{x}) \sum_{n = 0}^N s_{mn} \mathbf{e}_n(\mathbf{x}) 
- \varepsilon(\mathbf{x}) \int_0^T e^{-2t} \mathbf{E}_{tt}(\mathbf{x}, t) \Psi_m(t) \, d t.
\label{3.3}
\end{multline}

We now state a result for approximating the right-hand side of \eqref{3.3}:

\begin{Lemma}
Let $u \in H^2((0, T); L^2(\Omega))$. Let $\begin{bmatrix} u_0 & u_1 & \dots & u_N \end{bmatrix}^{\top} = \mathbb{F}^N[u]$. Then
\[
\lim_{N \to \infty} \sum_{m = 0}^\infty \left\|
    \sum_{n = 0}^N s_{mn} u_n(\mathbf{x}) 
    - \int_0^T e^{-2t} u_{tt}(\mathbf{x}, t) \Psi_m(t) \, dt
\right\|_{L^2(\Omega)}^2 = 0.
\]
\label{lem3.1}
\end{Lemma}

\begin{proof}
By Proposition~\ref{Prop_Trong}, we have
\begin{equation}
\lim_{N \to \infty} \left\|
    \sum_{n=0}^N \left\langle u(\cdot, \cdot), \Psi_n \right\rangle_{L^2_{e^{-2t}}(0, T)} \Psi_n''(t)
    - \frac{\partial^2 u}{\partial t^2}
\right\|^2_{L^2_{e^{-2t}}((0, T); L^2(\Omega))} = 0.
\label{3.6}
\end{equation}

Using Parseval's identity and the orthonormality of $\{\Psi_n\}$, we can write the left-hand side of \eqref{3.6} as
\begin{equation}
\sum_{m = 0}^\infty \int_{\Omega} \left| \left\langle
    \sum_{n=0}^N \left\langle u(\mathbf{x}, \cdot), \Psi_n \right\rangle \Psi_n''(t)
    - \frac{\partial^2 u(\mathbf{x}, \cdot)}{\partial t^2}, \Psi_m(t)
\right\rangle_{L^2_{e^{-2t}}(0, T)} \right|^2 d\mathbf{x}.
\label{3.7}
\end{equation}

The definitions of $s_{mn}$ and $u_n$ give
\[
\left\langle \sum_{n=0}^N \left\langle u(\mathbf{x}, \cdot), \Psi_n \right\rangle \Psi_n'', \Psi_m \right\rangle_{L^2_{e^{-2t}}(0, T)} = \sum_{n=0}^N s_{mn} u_n(\mathbf{x}),
\]
and
\[
\left\langle \frac{\partial^2 u(\mathbf{x}, \cdot)}{\partial t^2}, \Psi_m \right\rangle_{L^2_{e^{-2t}}(0, T)} = \int_0^T e^{-2t} u_{tt}(\mathbf{x}, t) \Psi_m(t) \, dt.
\]

Substituting into \eqref{3.7} and combining with \eqref{3.6} gives
\begin{multline*}
\left\|
    \sum_{n=0}^N \left\langle u(\cdot, \cdot), \Psi_n \right\rangle \Psi_n''(t)
    - \frac{\partial^2 u}{\partial t^2}
\right\|^2_{L^2_{e^{-2t}}((0, T); L^2(\Omega))} \\
= \sum_{m = 0}^\infty \int_{\Omega} \left| \sum_{n=0}^N s_{mn} u_n(\mathbf{x}) 
- \int_0^T e^{-2t} u_{tt}(\mathbf{x}, t) \Psi_m(t) \, dt \right|^2 d\mathbf{x},
\end{multline*}
which yields the desired convergence.
\end{proof}

Thanks to Lemma~\ref{lem3.1}, the right-hand side of \eqref{3.3} vanishes as $N \to \infty$, leading to the simplified equation
\begin{equation}
\nabla \times \left( \mu^{-1}(\mathbf{x}) \nabla \times \mathbf{e}_m(\mathbf{x}) \right)
+ \varepsilon(\mathbf{x}) \sum_{n = 0}^N s_{mn} \mathbf{e}_n(\mathbf{x}) 
= \mathbf{0}, \quad \mathbf{x} \in \Omega.
\label{3.4}
\end{equation}

This system represents a coupled family of elliptic equations in three spatial dimensions, with each $\mathbf{e}_m$ corresponding to a temporal mode of the electric field. The boundary conditions for $\mathbf{e}_m$ are derived from the measured data via time-projection as given in equations \eqref{Dir} and \eqref{Neu}, completing the full reduction of the original $(3+1)$-dimensional problem to a sequence of 3D systems.

To determine boundary conditions for each $\mathbf{e}_m$, we apply the Fourier operator $\mathbb{F}^N$ to the boundary data in \eqref{data}, yielding:
\begin{equation}
\mathbf{e}_m(\mathbf{x}) = \int_0^T e^{-2t} \mathbf{E}(\mathbf{x}, t) \Psi_m(t) \, dt
= \int_0^T e^{-2t} \mathbf{F}(\mathbf{x}, t) \Psi_m(t) \, dt =: \mathbf{f}_m(\mathbf{x}),
\label{Dir}
\end{equation}
and
\begin{equation}
\partial_\nu \mathbf{e}_m(\mathbf{x}) = \int_0^T e^{-2t} \partial_\nu \mathbf{E}(\mathbf{x}, t) \Psi_m(t) \, dt
= \int_0^T e^{-2t} \mathbf{G}(\mathbf{x}, t) \Psi_m(t) \, dt =: \mathbf{g}_m(\mathbf{x}).
\label{Neu}
\end{equation}

Combining the reduced system \eqref{3.4} with the boundary conditions \eqref{Dir} and \eqref{Neu}, we obtain the following system of equations for each $m = 0, \dots, N$:
\begin{equation}
\begin{cases}
\nabla \times \left( \mu^{-1}(\mathbf{x}) \nabla \times \mathbf{e}_m(\mathbf{x}) \right)
+ \varepsilon(\mathbf{x}) \sum_{n=0}^N s_{mn} \mathbf{e}_n(\mathbf{x}) = \mathbf{0}, & \mathbf{x} \in \Omega, \\
\mathbf{e}_m(\mathbf{x}) = \mathbf{f}_m(\mathbf{x}), & \mathbf{x} \in \partial\Omega, \\
\partial_\nu \mathbf{e}_m(\mathbf{x}) = \mathbf{g}_m(\mathbf{x}), & \mathbf{x} \in \partial\Omega.
\end{cases}
\label{3.11}
\end{equation}

\begin{remark}
Solving the system \eqref{3.11} for the vector of coefficients \[\mathbb{F}^N[\mathbf{E}] = \begin{bmatrix} \mathbf{e}_0 & \cdots & \mathbf{e}_N \end{bmatrix}^\top\] enables the reconstruction of the electric field $\mathbf{E}(\mathbf{x}, t)$ via the spectral expansion $\mathbb{S}^N[\mathbf{E}]$. In particular, the initial state $\mathbf{E}_0 = \mathbf{E}(\cdot, 0)$ can be directly recovered, which is essential for solving the inverse problem numerically.
\end{remark}

\begin{remark}
The Legendre-polynomial exponential basis $\{ \Psi_n \}_{n \geq 0}$ satisfies $\Psi_n'' \not\equiv 0$ for all $n \geq 0$. This property is essential for ensuring the accuracy of our numerical approximation. If a different basis $\{\phi_n\}_{n \geq 0}$ is chosen such that $\phi_n'' \equiv 0$ for some $n$, key components of the solution may be omitted. 

Notable examples include the standard Fourier basis and classical Legendre polynomials without the exponential weight. In such cases, the coefficients
\[
s_{mn} = \int_0^T \phi_n''(t) \phi_m(t) \, dt
\]
vanish for all $m$, causing the corresponding term $\mathbf{e}_n$ to disappear from equation~\eqref{3.4}. This results in a significant loss of information in the numerical reconstruction.

An alternative to the Legendre-polynomial exponential basis is the polynomial-exponential basis introduced in~\cite{Klibanov:jiip2017}, which also satisfies the condition $\phi_n'' \not\equiv 0$ and yields comparable numerical performance.
\label{rem31}
\end{remark}

Having reduced the original time-dependent Maxwell system to the static system~\eqref{3.11} through time-projection and Fourier expansion, we now focus on its numerical resolution. In the next section, we present a quasi-reversibility method for solving this system and establish a convergence theorem that includes a stability analysis in the presence of measurement noise.

\section{The quasi-reversibility method} \label{quasi}

Recall that $\mathbf{F}$ and $\mathbf{G}$ represent the measured boundary data defined in~\eqref{data}. Let $\mathbf{F}^*$ and $\mathbf{G}^*$ denote their corresponding noise-free versions, and let $\mathbf{F}^\delta$ and $\mathbf{G}^\delta$ represent the noisy measurements. The measurement noise is assumed to satisfy the bound
\begin{equation}
    \|\mathbf{F}^* - \mathbf{F}^{\delta}\|_{L^2_{e^{-2t}}((0, T); L^2(\partial \Omega)^3)}^2 + 
    \|\mathbf{G}^* - \mathbf{G}^{\delta}\|_{L^2_{e^{-2t}}((0, T); L^2(\partial \Omega)^3)}^2 \leq \delta^2,
    \label{noise}
\end{equation}
where $\delta > 0$ denotes the noise level.
The corresponding noise-free and noisy Fourier coefficients of the boundary data used in~\eqref{3.11} are denoted by $\mathbf{f}_m^*$ and $\mathbf{g}_m^*$, and $\mathbf{f}_m^\delta$ and $\mathbf{g}_m^\delta$, respectively, for each $m \in \{0, \dots, N\}$. Then, by Parseval's identity and the bound in~\eqref{noise}, we obtain
\begin{equation}
    \sum_{m = 0}^N \left[
        \|\mathbf{f}_m^* - \mathbf{f}_m^\delta\|_{L^2(\partial \Omega)^3}^2 + 
        \|\mathbf{g}_m^* - \mathbf{g}_m^\delta\|_{L^2(\partial \Omega)^3}^2
    \right] \leq \delta^2.
    \label{noise_m}
\end{equation}

The boundary value problem in equation~\eqref{3.11} is over-determined and, in general, may not possess a solution, especially when the boundary data $\mathbf{f}_m$ and $\mathbf{g}_m$ are replaced by their noisy versions $\mathbf{f}_m^\delta$ and $\mathbf{g}^\delta_m$ respectively. To address this challenge, we seek the best-fit solution by minimizing the following Tikhonov-type regularized functional:
\begin{multline}
	J_{N, \delta, \epsilon}(\mathbf{V}) = \sum_{m=0}^N \Bigg[ \int_{\Omega} \left| \nabla \times (\mu^{-1}(\x)\nabla \times \mathbf{v}_m(\mathbf{x})) + \varepsilon(\x)\sum_{n=0}^N s_{mn} \mathbf{v}_n(\mathbf{x})\right|^2 \, d\mathbf{x} \\
	+ \int_{\partial \Omega} \left| \mathbf{v}_m(\mathbf{x}) - \mathbf{f}_m^\delta(\mathbf{x}) \right|^2 \, d\sigma(\mathbf{x}) 
	+ \int_{\partial \Omega} \left| \partial_{\nu} \mathbf{v}_m(\mathbf{x}) - \mathbf{g}_m^\delta(\mathbf{x}) \right|^2 \, d\sigma(\mathbf{x}) 
	+ \epsilon \| \mathbf{v}_m \|_{H^3(\Omega)^3}^2 \Bigg],
	\label{4.1}
\end{multline}
where $\mathbf{V} = \begin{bmatrix} \mathbf{v}_0 & \mathbf{v}_1 & \dots & \mathbf{v}_N \end{bmatrix}^\top \in H^3(\Omega)^{3 \times (N+1)}$, and $\epsilon > 0$ is a regularization parameter.
We minimize $J_{N, \delta, \epsilon}$ over the Hilbert space
\[
	H = \left\{ \mathbf{V} = \begin{bmatrix} \mathbf{v}_0 & \mathbf{v}_1 & \dots & \mathbf{v}_N \end{bmatrix}^\top : \mathbf{v}_m \in H^3(\Omega)^3 \text{ for } m = 0, \dots, N \right\},
\]
and denote by $\mathbf{V}_{\min}^{N, \delta, \epsilon} = \begin{bmatrix} \mathbf{v}_0^{N, \delta, \epsilon} & \mathbf{v}_1^{N, \delta, \epsilon} & \dots & \mathbf{v}_N^{N, \delta, \epsilon} \end{bmatrix}^\top$ the unique minimizer of the functional.

\begin{remark}[The existence of ${\bf V}_{\min}^{N, \delta, \epsilon}$]
The functional $J_{N, \delta, \epsilon}$ defined in \eqref{4.1} admits a unique minimizer in the space $H$. This is because $J_{N, \delta, \epsilon}$ is strictly convex, coercive, and continuous on $H$. The strict convexity stems from the regularization term $\epsilon \| \mathbf{v}_m \|^2_{H^3(\Omega)^3}$, which enforces smoothness and ensures that the second variation of $J_{N, \delta, \epsilon}$ is strictly positive. Coercivity is also a consequence of this term, as it guarantees that $J_\epsilon(\mathbf{V}) \to \infty$ whenever $\| \mathbf{V} \|_{H^3(\Omega)^{3 \times (N + 1)}} \to \infty$. Along with the weak lower semicontinuity of $J_{N, \delta, \epsilon}$, these properties collectively ensure the existence and uniqueness of a minimizer $\mathbf{V}_{\min}^{N, \delta,\epsilon} \in H$.
\end{remark}

As described in Remark \ref{rem31}, our time-dimensional reduction strategy employs the space-time expansion $\mathbb{S}^N[\mathbf{V}^{N, \delta, \epsilon}_{\mathrm{min}}]$, defined in \eqref{SN1}, to approximate the true minimum-norm solution $\mathbf{E}^*$ to \eqref{3_1}.
 While this approximation is computationally practical and often accurate in practice, it naturally raises a theoretical question regarding its reliability. The remainder of this section is devoted to addressing this question affirmatively.

The following is our main convergence result.

\begin{Theorem}[ Convergence of the quasi-reversibility method for noisy data]
Suppose the noisy boundary data $\mathbf{F}^\delta$ and $\mathbf{G}^\delta$ satisfy the noise constraint~\eqref{noise}, and assume that $\delta^2 = o(\epsilon)$ as $\delta, \epsilon \to 0$ where $\epsilon$ is the regularization parameter. Then:

\begin{enumerate}
    \item There exists an integer $N(\delta) > 0$ such that for all $N > N(\delta)$, the following estimate holds:
    \begin{equation}
        \sum_{m=0}^N \int_{\Omega} \left|  
        \varepsilon(\mathbf{x}) \sum_{n=1}^N s_{mn} \mathbf{e}_n^*(\mathbf{x}) 
        - \varepsilon(\mathbf{x}) \int_0^T e^{-2t} \partial_{tt} \mathbf{E}^*(\mathbf{x}, t) \Psi_m(t) \, dt 
        \right|^2 d\mathbf{x} \leq \delta^2.
        \label{4.10}
    \end{equation}

    \item Define the parameter set
    \begin{equation}
        \Theta := \left\{ (N, \delta, \epsilon): N \geq N(\delta), \, \delta^2 = o(\epsilon) \right\}.
        \label{5.5}
    \end{equation}
    Then, as $(N, \delta, \epsilon) \in \Theta$ tends to $(\infty, 0^+, 0^+)$, we have the following convergence results:
    \begin{equation}
        \lim_{\Theta \ni (N, \delta, \epsilon) \to (\infty, 0^+, 0^+)} \left\| \mathbb{S}^N[\mathbf{V}_{\min}^{N, \delta, \epsilon}] - \mathbf{E}^* \right\|_{L^2((0, T); H^2(\Omega)^3)} = 0,
        \label{5.6}
    \end{equation}
    \begin{equation}
        \lim_{\Theta \ni (N, \delta, \epsilon) \to (\infty, 0^+, 0^+)} \left\| \partial_t \mathbb{S}^N[\mathbf{V}_{\min}^{N, \delta, \epsilon}] - \partial_t \mathbf{E}^* \right\|_{L^2((0, T); L^2(\Omega)^3)} = 0,
        \label{5.7}
    \end{equation}
    \begin{equation}
        \lim_{\Theta \ni (N, \delta, \epsilon) \to (\infty, 0^+, 0^+)} \left\| \mathbb{S}^N[\mathbf{V}_{\min}^{N, \delta, \epsilon}](\cdot, 0) - \mathbf{E}^*(\cdot, 0) \right\|_{L^2(\Omega)^3} = 0.
        \label{5.8888}
    \end{equation}
\end{enumerate}
\label{thm41}
\end{Theorem}

\begin{proof}

The existence of $N(\delta)$ and the estimate in~\eqref{4.10} follow from Lemma~\ref{lem3.1} directly.  
We now focus on proving the second part of the theorem. The argument is divided into several steps.

{\bf Step 1.}
We begin by applying the time-projection operator $\mathbb{P}^N$ to both sides of the exact model~\eqref{3_1}. This leads to the following equation for each Fourier mode $\mathbf{e}_m^*(\mathbf{x})$ of $\mathbb{F}^N[\mathbf{E}^*]$:
\begin{equation}
\nabla \times \left( \mu^{-1}(\mathbf{x}) \nabla \times \mathbf{e}_m^*(\mathbf{x}) \right)
+ \varepsilon(\mathbf{x}) \int_0^T e^{-2t} \mathbf{E}_{tt}^*(\mathbf{x}, t) \Psi_m(t) \, dt
= \mathbf{0}, \quad \mathbf{x} \in \Omega.
\label{4.8}
\end{equation}

Next, using the minimality of $\mathbf{V}_{\min}^{N, \delta, \epsilon}$, the noise estimate~\eqref{noise_m}, and the inequality $(a + b)^2 \leq 2a^2 + 2b^2$, we obtain:
\begin{align}
&J_{N, \delta, \epsilon}(\mathbf{V}_{\min}^{N, \delta, \epsilon}) 
\leq J_{N, \delta, \epsilon}(\mathbb{F}^N[\mathbf{E}^*]) \notag \\
&\quad \leq \sum_{m=0}^N \Bigg[
\int_{\Omega} \left| \nabla \times (\mu^{-1}(\x) \nabla \times \mathbf{e}_m^*(\x)) + \varepsilon(\x) \sum_{n=1}^N s_{mn} \mathbf{e}_n^*(\x) \right|^2 \, d\mathbf{x} \\
& \hspace{1cm} +  \delta^2 + \epsilon \| \mathbf{e}_m^*(\x) \|_{H^3(\Omega)^3}^2
\Bigg] \notag \\
&\quad\leq \sum_{m=0}^N \Bigg[
2\int_{\Omega} \left| \nabla \times (\mu^{-1}(\x) \nabla \times \mathbf{e}_m^*(\x)) + \varepsilon(\x) \int_0^T e^{-2t} \partial_{tt} \mathbf{E}^*(\x, t) \, dt \right|^2 \, d\mathbf{x} \notag \\
&\hspace{1cm} + 2\int_{\Omega} \left| -\varepsilon(\x) \int_0^T e^{-2t} \partial_{tt} \mathbf{E}^*(\x, t) \, dt + \varepsilon(\x) \sum_{n=1}^N s_{mn} \mathbf{e}_n^*(\x) \right|^2 \, d\mathbf{x} \\
& \hspace{1cm} + \delta^2 + \epsilon \| \mathbf{e}_m^* \|_{H^3(\Omega)^3}^2
\Bigg]
\label{4.7}
\end{align}
for all $N \geq N(\delta)$.
Using \eqref{4.10}, \eqref{4.8}, and \eqref{4.7}, we obtain the bound:
\begin{equation}
	J_{N, \delta, \epsilon}(\mathbf{V}_{\min}^{N, \delta, \epsilon}) \leq 3\delta^2 + \epsilon \sum_{m = 0}^N \| \mathbf{e}_m^* \|^2_{H^3(\Omega)^3}.
	\label{5.8}
\end{equation}
Combining this with \eqref{4.1} and using \eqref{2.4}, we get:
\begin{equation}
	\| \mathbb{S}^N[\mathbf{V}_{\min}^{N, \delta, \epsilon}] \|^2_{L^2_{e^{-2t}}((0, T); H^3(\Omega)^3)} = \sum_{m=0}^N \| \mathbf{v}_m \|^2_{H^3(\Omega)^3} \leq \frac{3\delta^2}{\epsilon} + \sum_{m=0}^N \| \mathbf{e}_m^* \|^2_{H^3(\Omega)^3} \leq C,
	\label{5.9}
\end{equation}
provided that $\delta^2 = o(\epsilon)$ and $N > N(\delta)$.

\textbf{Step 2.} We next establish bounds for $\partial_t \mathbb{S}^N[\mathbf{V}_{\min}^{N, \delta, \epsilon}]$ and $\partial_{tt} \mathbb{S}^N[\mathbf{V}_{\min}^{N, \delta, \epsilon}]$. More precisely,
\begin{align}
\| \varepsilon & \, \partial_{tt} \mathbb{S}^N[\mathbf{V}_{\min}^{N, \delta, \epsilon}] \|_{L^2_{e^{-2t}}((0, T); L^2(\Omega)^3)}
\notag
\\
&\leq \left\| \varepsilon \, \partial_{tt} \mathbb{S}^N[\mathbf{V}_{\min}^{N, \delta, \epsilon}]
+ \nabla \times (\mu^{-1} \nabla \times \mathbb{S}^N[\mathbf{V}_{\min}^{N, \delta, \epsilon}]) \right\|_{L^2_{e^{-2t}}((0, T); L^2(\Omega)^3)} \notag
\\
&\hspace{4cm}
 + \left\| \nabla \times (\mu^{-1} \nabla \times \mathbb{S}^N[\mathbf{V}_{\min}^{N, \delta, \epsilon}]) \right\|_{L^2_{e^{-2t}}((0, T); L^2(\Omega)^3)} \notag \\
&\leq \left( \sum_{m=0}^N \left| \left\langle \varepsilon \, \partial_{tt} \mathbb{S}^N[\mathbf{V}_{\min}^{N, \delta, \epsilon}] + \nabla \times (\mu^{-1} \nabla \times \mathbb{S}^N[\mathbf{V}_{\min}^{N, \delta, \epsilon}]), \Psi_m \right\rangle_{L^2_{e^{-2t}}((0, T); L^2(\Omega)^3)} \right|^2 \right)^{\frac{1}{2}} \notag \\
&\hspace{4cm} + C \| \mathbb{S}^N[\mathbf{V}_{\min}^{N, \delta, \epsilon}] \|_{L^2_{e^{-2t}}((0, T); H^3(\Omega)^3)},
\label{5.10}
\end{align}
where Parseval's identity has been applied.
A direct computation shows that
\begin{multline}
\sum_{m=0}^N \left| \left\langle \varepsilon \, \partial_{tt} \mathbb{S}^N[\mathbf{V}_{\min}^{N, \delta, \epsilon}] + \nabla \times \left( \mu^{-1} \nabla \times \mathbb{S}^N[\mathbf{V}_{\min}^{N, \delta, \epsilon}] \right), \Psi_m \right\rangle \right|^2 \\
= \sum_{m=0}^N \int_{\Omega} \left| \sum_{n=0}^N \varepsilon s_{mn} \mathbf{v}_n^{N, \delta, \epsilon} + \nabla \times \left( \mu^{-1} \nabla \times \mathbf{v}_m^{N, \delta, \epsilon} \right) \right|^2 \, d\mathbf{x}
\leq J_{N, \delta, \epsilon}(\mathbf{V}_{\min}^{N, \delta, \epsilon}).
\label{5.11}
\end{multline}
Combining inequalities \eqref{5.8}, \eqref{5.9}, \eqref{5.10}, and \eqref{5.11}, we obtain the uniform estimate
\begin{equation}
\left\| \varepsilon \, \partial_{tt} \mathbb{S}^N[\mathbf{V}_{\min}^{N, \delta, \epsilon}] \right\|_{L^2_{e^{-2t}}((0, T); L^2(\Omega)^3)} \leq C.
\label{5.12}
\end{equation}

This result provides a bound on the second time derivative of $\mathbb{S}^N[\mathbf{V}_{\min}^{N, \delta, \epsilon}]$. To control the first time derivative, we apply Lemma \ref{lem4.2}, which implies that
\begin{multline}
\left\| \varepsilon \, \partial_t \mathbb{S}^N[\mathbf{V}_{\min}^{N, \delta, \epsilon}] \right\|_{L^2_{e^{-2t}}((0, T); L^2(\Omega)^3)} 
\\
\leq C(\left\| \varepsilon \,  \mathbb{S}^N[\mathbf{V}_{\min}^{N, \delta, \epsilon}] \right\|_{L^2_{e^{-2t}}((0, T); L^2(\Omega)^3)} + \left\| \varepsilon \, \partial_{tt} \mathbb{S}^N[\mathbf{V}_{\min}^{N, \delta, \epsilon}] \right\|_{L^2_{e^{-2t}}((0, T); L^2(\Omega)^3)}) \leq C
\label{5.13}
\end{multline}

\textbf{Step 3.}  
Using the uniform bounds established in \eqref{5.9}, \eqref{5.12}, and \eqref{5.13}, we conclude that the family  
\[
\left\{ \mathbb{S}^N[\mathbf{V}_{\min}^{N, \delta, \epsilon}] : (N, \delta, \epsilon) \in \Theta \right\}
\]  
is bounded in both $L^2((0, T); H^3(\Omega)^3)$ and $H^2((0, T); L^2(\Omega)^3)$, where $\Theta$ is the set defined in \eqref{5.5}. By compactness of Sobolev embeddings, there exists a subsequence, still denoted by $\mathbb{S}^N[\mathbf{V}_{\min}^{N, \delta, \epsilon}]$, that converges as $(N, \delta, \epsilon) \in \Theta \to (\infty, 0^+, 0^+)$ in the following sense:
\[
\partial_{tt} \mathbb{S}^N[\mathbf{V}_{\min}^{N, \delta, \epsilon}] \rightharpoonup \mathbf{z} \quad \text{weakly in } L^2((0, T); H^2(\Omega)^3),
\]
\[
\mathbb{S}^N[\mathbf{V}_{\min}^{N, \delta, \epsilon}] \to \mathbf{z} \quad \text{strongly in } L^2((0, T); H^2(\Omega)^3),
\]
\[
\partial_t \mathbb{S}^N[\mathbf{V}_{\min}^{N, \delta, \epsilon}] \to \partial_t \mathbf{z} \quad \text{strongly in } L^2((0, T); L^2(\Omega)^3),
\]
\[
\mathbb{S}^N[\mathbf{V}_{\min}^{N, \delta, \epsilon}] \to \mathbf{z} \quad \text{strongly in } L^2_{e^{-2t}}((0, T); L^2(\partial \Omega)^3),
\]
\[
\partial_\nu \mathbb{S}^N[\mathbf{V}_{\min}^{N, \delta, \epsilon}] \to \partial_\nu \mathbf{z} \quad \text{strongly in } L^2_{e^{-2t}}((0, T); L^2(\partial \Omega)^3).
\]

We now verify that the limiting function $\mathbf{z}$ solves the original problem \eqref{3_1}. In fact, we have the inequality:
\begin{multline}
\int_{\Omega_T} \left| \nabla \times \left( \mu^{-1}(\mathbf{x}) \nabla \times \mathbf{z} \right) + \varepsilon(\mathbf{x}) \partial_{tt} \mathbf{z} \right|^2 d\mathbf{x} \, dt
+ \int_{\Gamma_T} \left| \mathbf{z} - \mathbf{F}^* \right|^2 d\sigma(\mathbf{x}) \, dt \\
+ \int_{\Gamma_T} \left| \partial_\nu \mathbf{z} - \mathbf{G}^* \right|^2 d\sigma(\mathbf{x}) \, dt
\leq \liminf_{\Theta \ni (N, \delta, \epsilon) \to (\infty, 0^+, 0^+)} J_{N, \delta, \epsilon}(\mathbf{V}_{\min}^{N, \delta, \epsilon}),
\label{5.14}
\end{multline}
where we used the lower semi-continuity of convex functionals and the convergence results above.
Substituting the bound from \eqref{5.8} into \eqref{5.14}, and recalling that $\delta^2 = o(\epsilon)$, we obtain
\begin{multline*}
\int_{\Omega_T} \left| \nabla \times \left( \mu^{-1}(\mathbf{x}) \nabla \times \mathbf{z} \right) + \varepsilon(\mathbf{x}) \partial_{tt} \mathbf{z} \right|^2 d\mathbf{x} \, dt
+ \int_{\Gamma_T} \left| \mathbf{z} - \mathbf{F}^* \right|^2 d\sigma(\mathbf{x}) \, dt \\
+ \int_{\Gamma_T} \left| \partial_\nu \mathbf{z} - \mathbf{G}^* \right|^2 d\sigma(\mathbf{x}) \, dt \\
\leq \liminf_{\Theta \ni (N, \delta, \epsilon) \to (\infty, 0^+, 0^+)} \left[ 3\delta^2 + \epsilon \| \mathbf{E}^* \|^2_{L^2_{e^{-2t}}((0, T); H^3(\Omega)^3)} \right] = 0.
\end{multline*}
Therefore, $\mathbf{z}$ satisfies the original model \eqref{3_1}.

\textbf{Step 4.}  
We now demonstrate that $\mathbf{z} = \mathbf{E}^*$. From the strong convergence of $\mathbb{S}^N[\mathbf{V}_{\min}^{N, \delta, \epsilon}]$ to $\mathbf{z}$ and the bound \eqref{5.9}, together with the assumption $\delta^2 = o(\epsilon)$, we obtain
\begin{align*}
\| \mathbf{z} \|^2_{L^2_{e^{-2t}}((0, T); H^3(\Omega)^3)} 
&= \lim_{\Theta \ni (N, \delta, \epsilon) \to (\infty, 0^+, 0^+)} 
\| \mathbb{S}^N[\mathbf{V}_{\min}^{N, \delta, \epsilon}] \|^2_{L^2_{e^{-2t}}((0, T); H^3(\Omega)^3)} \\
&\leq \lim_{\delta^2 = o(\epsilon) \to 0} \left( \frac{3\delta^2}{\epsilon} + \sum_{m=0}^N \| \mathbf{e}_m^* \|^2_{H^3(\Omega)^3} \right) \\
&\leq \| \mathbf{E}^* \|^2_{L^2_{e^{-2t}}((0, T); H^3(\Omega)^3)}.
\end{align*}

Since $\mathbf{z}$ belongs to the admissible set and achieves the minimum norm, the uniqueness of the minimizer implies that $\mathbf{z} = \mathbf{E}^*$.  
This completes the proof of convergence estimates \eqref{5.6} and \eqref{5.7}.  
Finally, estimate \eqref{5.8888} follows directly from \eqref{5.7} by the standard trace theorem.

\end{proof}

\section{Numerical study}\label{sec6}

In our numerical simulations, we set the computational domain $\Omega = (-1, 1)^3$ and the final time $T = 2.5$. The electric permittivity is taken to be constant, $\varepsilon(\mathbf{x}) = 1$, while the magnetic permeability $\mu(\mathbf{x})$ is defined as
\[
\mu(\mathbf{x}) = 
\begin{cases}
\displaystyle \frac{1}{1 + 0.1 \exp\left(-\dfrac{|\mathbf{x}|^2}{0.5^2 - |\mathbf{x}|^2}\right)}, & \text{if } |\mathbf{x}| < 0.5, \\
1, & \text{otherwise}.
\end{cases}
\]
To generate the simulated data, we numerically solve equation~\eqref{Maxwell} using a prescribed initial condition $\mathbf{E}_0$. Since solving Maxwell's equations on the entire space $\mathbb{R}^3$ is not practical when the magnetic permeability $\mu$ varies spatially, we instead consider a bounded computational domain $G = (-2.5, 2.5)^3 \supset \Omega$ and employ an explicit time-stepping scheme. Provided that the final time $T$ is chosen not too large so that the electromagnetic field does not interact with the boundary of $G$, the resulting solution within $G$, and in particular within $\Omega$, serves as a good approximation of the restriction of the true solution in the full space. When solving the forward problem to generate the data, we impose the additional condition $\partial_t \mathbf{E}(\mathbf{x}, 0) = 0$ to ensure well-posedness of the simulation. However, this initial velocity information is deliberately excluded from the inverse problem formulation to reflect the practical limitation of incomplete data and test our reconstruction method's robustness under under-determined settings.

After computing the electric field $\mathbf{E}^*$ by solving the forward problem, we extract the exact boundary data $\mathbf{F}^*$ and $\mathbf{G}^*$ by evaluating $\mathbf{F}^*(\mathbf{x}, t) = \mathbf{E}^*(\mathbf{x}, t)$ and $\mathbf{G}^*(\mathbf{x}, t) = (\nabla \times \mathbf{E}^*(\mathbf{x}, t)) \times \nu$ for $(\mathbf{x}, t) \in \Gamma_T$. To simulate measurement errors, we perturb these data with additive noise to obtain the noisy boundary observations $\mathbf{F}^\delta$ and $\mathbf{G}^\delta$, corresponding to a relative noise level of $\delta = 10\%$. The noisy data are generated according to
\[
\mathbf{F}^\delta = \mathbf{F}^*(1 + \delta  \texttt{rand}), \quad
\mathbf{G}^\delta = \mathbf{G}^*(1 + \delta  \texttt{rand}),
\]
where \texttt{rand} is a random function returning uniformly distributed values in the interval $[-1, 1]$.

Remark~\ref{rem31} and Theorem~\ref{thm41}, especially \eqref{5.8888}, motivate the following procedure, summarized in Algorithm~\ref{alg}, for computing the initial electric field $\mathbf{E}_0$.
We implement the proposed Algorithm \ref{alg} using a finite difference scheme. The computational domain $\Omega = (-1, 1)^3$ is discretized using a uniform Cartesian grid consisting of $20 \times 20 \times 20$ points. For the time interval $[0, T]$ with $T = 2.5$, we employ a uniform temporal discretization with 73 time steps, resulting in a time step size of $\Delta t = \frac{2.5}{72} \approx 0.0347$. This spatial-temporal discretization is used consistently throughout the numerical simulation to approximate spatial derivatives and advance the electric field in time.

\begin{algorithm}[ht]
\caption{\label{alg} Time-Dimensional Reduction Method for Reconstructing the Initial Electric Field}
\begin{algorithmic}[1]

\Input Measured boundary data $\mathbf{F}^\delta(\mathbf{x}, t)$ and $\mathbf{G}^\delta(\mathbf{x}, t)$, noise level $\delta$.
\Output Approximate initial field $\mathbf{E}_{\mathrm{comp}}(\mathbf{x}) \approx \mathbf{E}^*(\mathbf{x}, 0)$.

\State Choose a regularization parameter $\epsilon$ and truncation order $N$. \label{s1}
\State Construct the basis functions $\{\Psi_n(t)\}_{n=0}^N$.
\State Compute the Fourier modes of the boundary data and source term:
\begin{equation*}
    \begin{array}{rcl}
          \mathbf{f}_m^\delta(\mathbf{x}) &:= \ds\int_0^T e^{-2t} \mathbf{F}^\delta(\mathbf{x}, t) \Psi_m(t) \, dt, \\
    \mathbf{g}_m^\delta(\mathbf{x}) &:= \ds\int_0^T e^{-2t} \mathbf{G}^\delta(\mathbf{x}, t) \Psi_m(t) \, dt. \\
    \end{array}
    \quad \mbox{for } \x \in \Omega, m \in \{0, 1, \dots, N\}.
\end{equation*}
\State Minimize the cost functional $J_{N, \delta, \epsilon}(\mathbf{V})$ defined in~\eqref{4.1}; let ${\bf V}_{\mathrm{comp}} = \mathbf{V}_{\min}^{N, \delta, \epsilon}$ denote the obtained minimizer. \label{s4}

\State Reconstruct the time-dependent approximation:
\[
\mathbb{S}^N[\mathbf{V}_{\min}^{N, \delta, \epsilon}](\mathbf{x}, t) := \sum_{n = 0}^N \mathbf{v}_n^{N, \delta, \epsilon}(\mathbf{x}) \Psi_n(t) 
\quad (\x, t) \in \Omega_T.
\]

\State Output the approximate initial field:
\[
\mathbf{E}_{\mathrm{comp}}(\mathbf{x}) := \mathbb{S}^N[\mathbf{V}_{\min}^{N, \delta, \epsilon}](\mathbf{x}, 0).
\]
\end{algorithmic}
\end{algorithm}

We determine the regularization parameter $\epsilon$ and the truncation order $N$ in Step~\ref{s1} of Algorithm~\ref{alg} through a trial-and-error process. Specifically, we select Test~1 as a reference case and manually tune the values of $N$ and $\epsilon$ to obtain satisfactory reconstruction results. Once acceptable performance is achieved for this reference test, we fix these parameters and apply the same values for all subsequent experiments. In our implementation, we choose $N = 15$ and $\epsilon = 10^{-6}$.

To minimize the functional $J_{N, \delta, \epsilon}$ in Step~\ref{s4} of Algorithm~\ref{alg}, we employ a linear algebra routine available in MATLAB. Since $J_{N, \delta, \epsilon}$ has a least-squares form involving only linear expressions of the unknowns, the minimization reduces to solving a large linear system. However, due to the high dimensionality of our 3D inverse problem and the fact that the electric field $\mathbf{E}$ has three vector components, standard solvers such as \texttt{lsqlin} are not suitable in our case, as they require more memory than is available on our machine. Instead, we use MATLAB's \texttt{pcg} (preconditioned conjugate gradient) method to efficiently solve the system. For further details on the usage of \texttt{pcg}, we refer the reader to the MATLAB documentation.

The remaining steps of Algorithm~\ref{alg} are straightforward to implement and do not pose significant computational difficulties. Once the basis functions are constructed and the Fourier modes of the data are computed via numerical quadrature, the reconstruction of the space-time approximation and the evaluation of the initial field follow directly from the definition of the truncated expansion $\mathbb{S}^N$. These steps involve basic operations such as function evaluations, summations, and interpolation on the computational grid.

\subsection*{Test 1}

The true initial condition is given by the vector field \[\mathbf{E}^{\rm true}_0(\mathbf{x}) = (E^{\rm true}_1, E^{\rm true}_2, E^{\rm true}_3)\] where each component is defined as follows:
\[
E^{\rm true}_1(\mathbf{x}) =
\begin{cases}
1, & \text{if } (x - 0.4)^2 + y^2 + (z + 0.3)^2 < 0.35^2, \\
0, & \text{otherwise},
\end{cases}
\]
\[
E^{\rm true}_2(\mathbf{x}) =
\begin{cases}
1, & \text{if } 0.4^2 < x^2 + z^2 < 0.8^2 \text{ and } |y| < 0.8, \\
0, & \text{otherwise},
\end{cases}
\]
\[
E^{\rm true}_3(\mathbf{x}) =
\begin{cases}
1, & \text{if } \max\{0.4x^2, (y - 0.55)^2 + (z - 0.3)^2\} < 0.3^2, \\
0, & \text{otherwise}.
\end{cases}
\]
The true initial electric field $\mathbf{E}^{\rm true}_0$ consists of three geometrically distinct components: a solid sphere centered at $(0.4, 0, -0.3)$ for $E_1^{\rm true}$, a horizontal cylindrical shell with a central void aligned along the $y$-axis for $E_2^{\rm true}$, and a small horizontally oriented solid cylinder offset in the $y$-direction for $E_3^{\rm true}$.

\begin{figure}
    \centering
	\subfloat[]{\includegraphics[height = .2\textwidth,width = .3\textwidth]{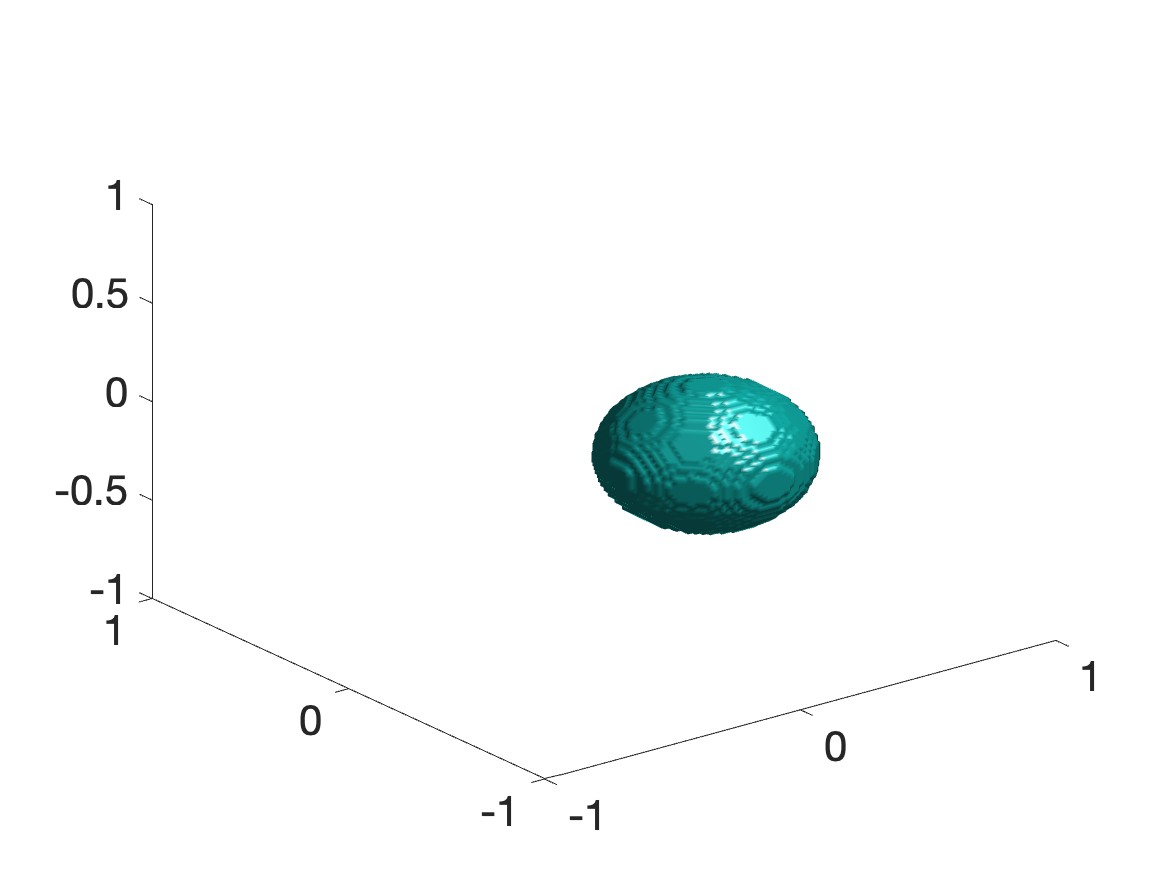}}
	\hfill
	\subfloat[]{\includegraphics[height = .2\textwidth,width = .3\textwidth]{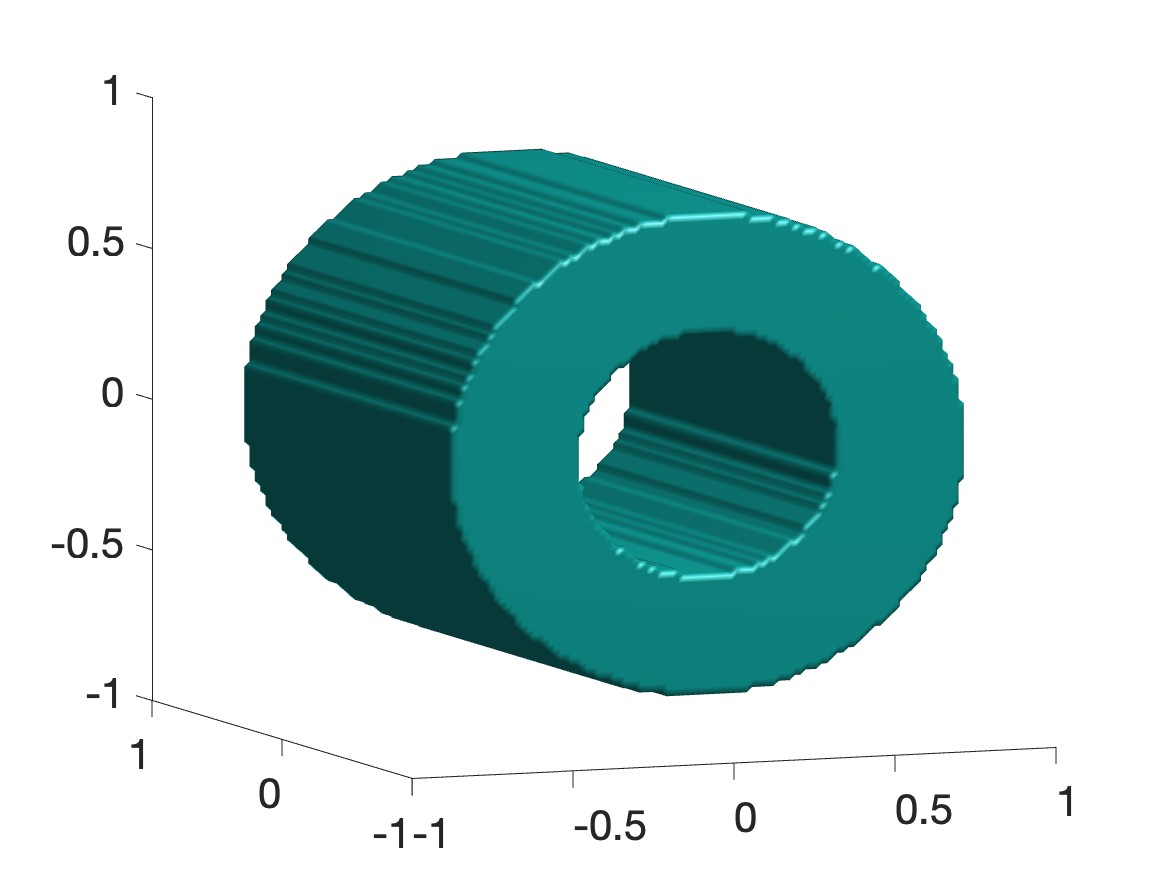}}
	\hfill
	\subfloat[]{\includegraphics[height = .2\textwidth,width = .3\textwidth]{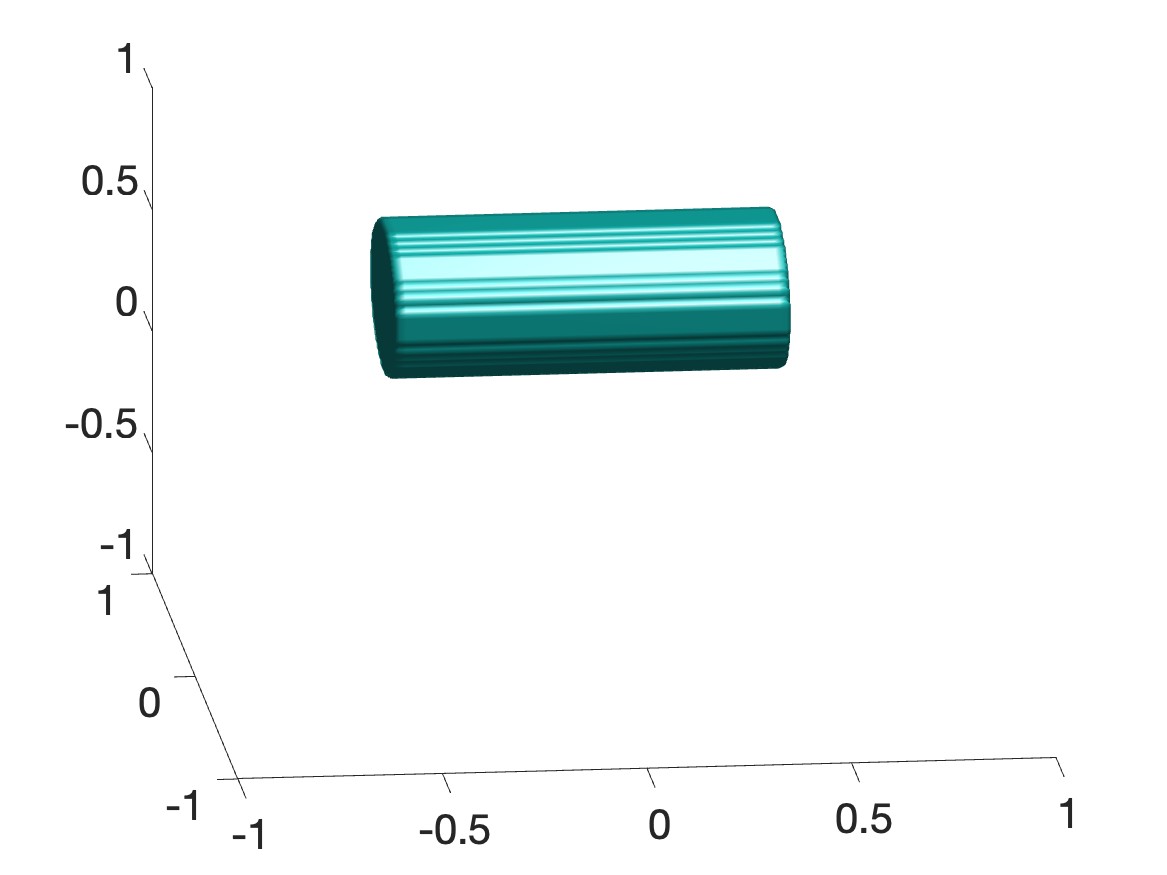}}
	
	\subfloat[]{\includegraphics[height = .2\textwidth,width = .3\textwidth]{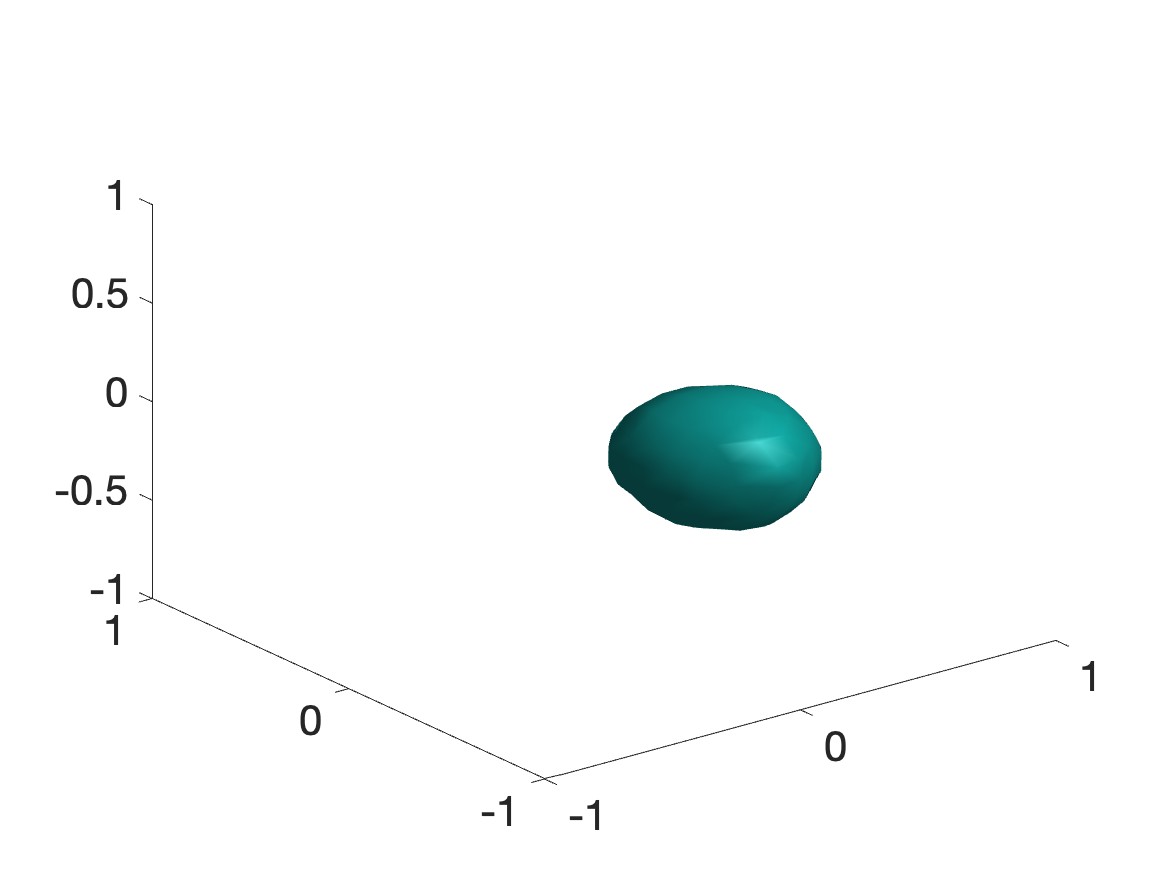}}
	\hfill
	\subfloat[]{\includegraphics[height = .2\textwidth,width = .3\textwidth]{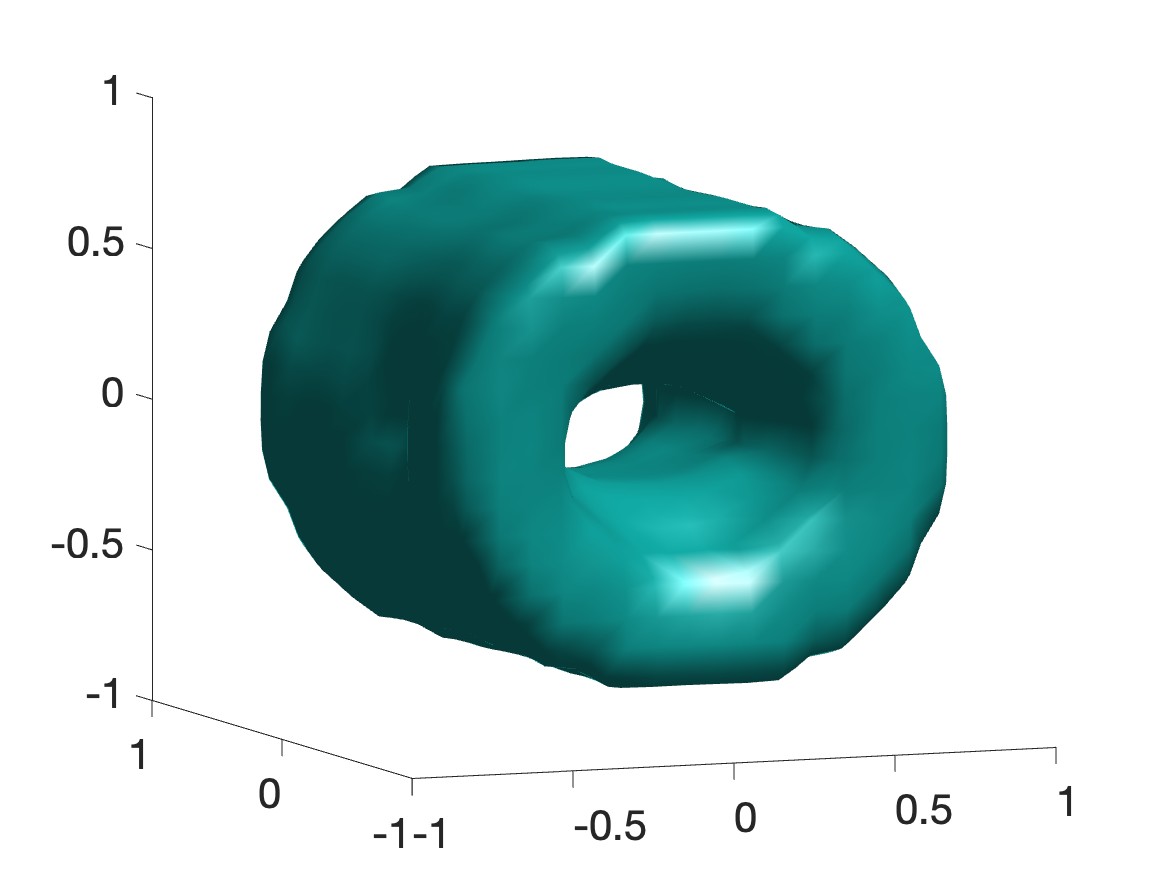}}
	\hfill
	\subfloat[]{\includegraphics[height = .2\textwidth,width = .3\textwidth]{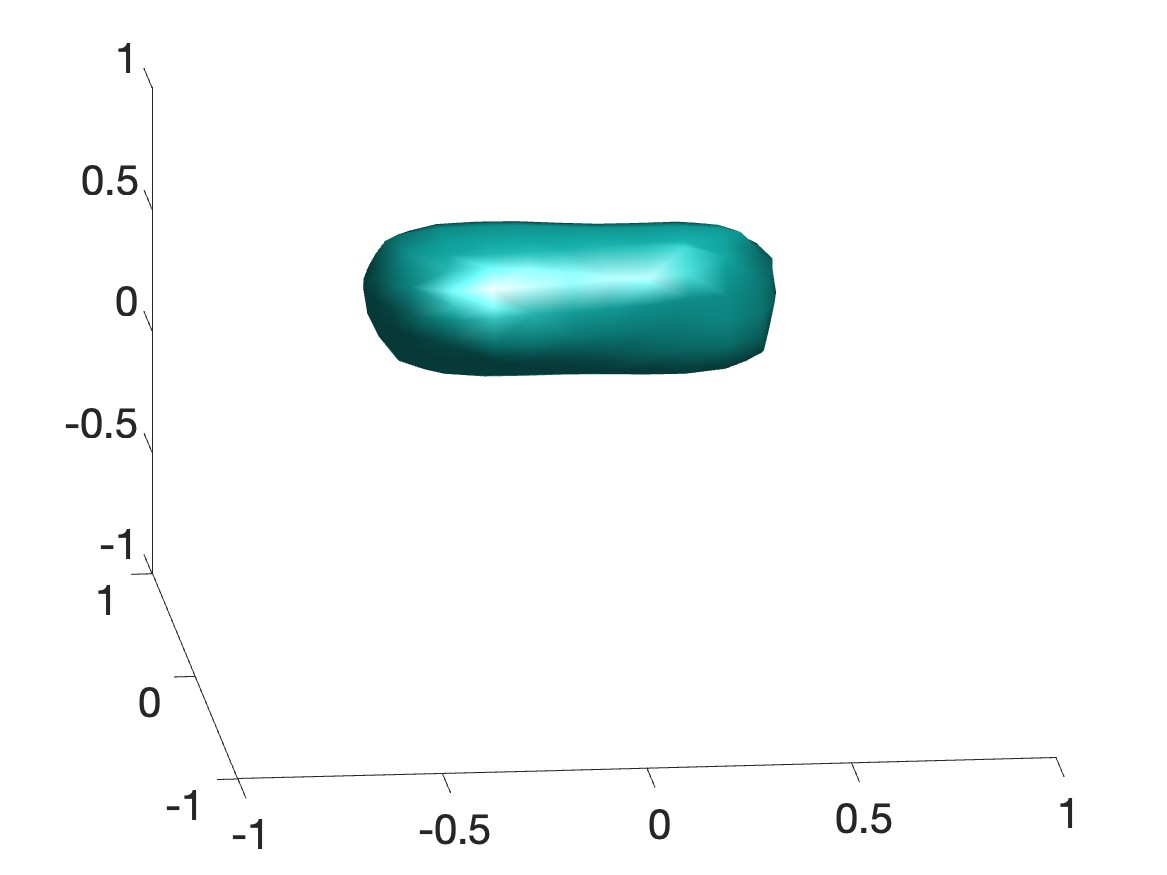}}
	
	\subfloat[]{\includegraphics[height = .2\textwidth,width = .3\textwidth]{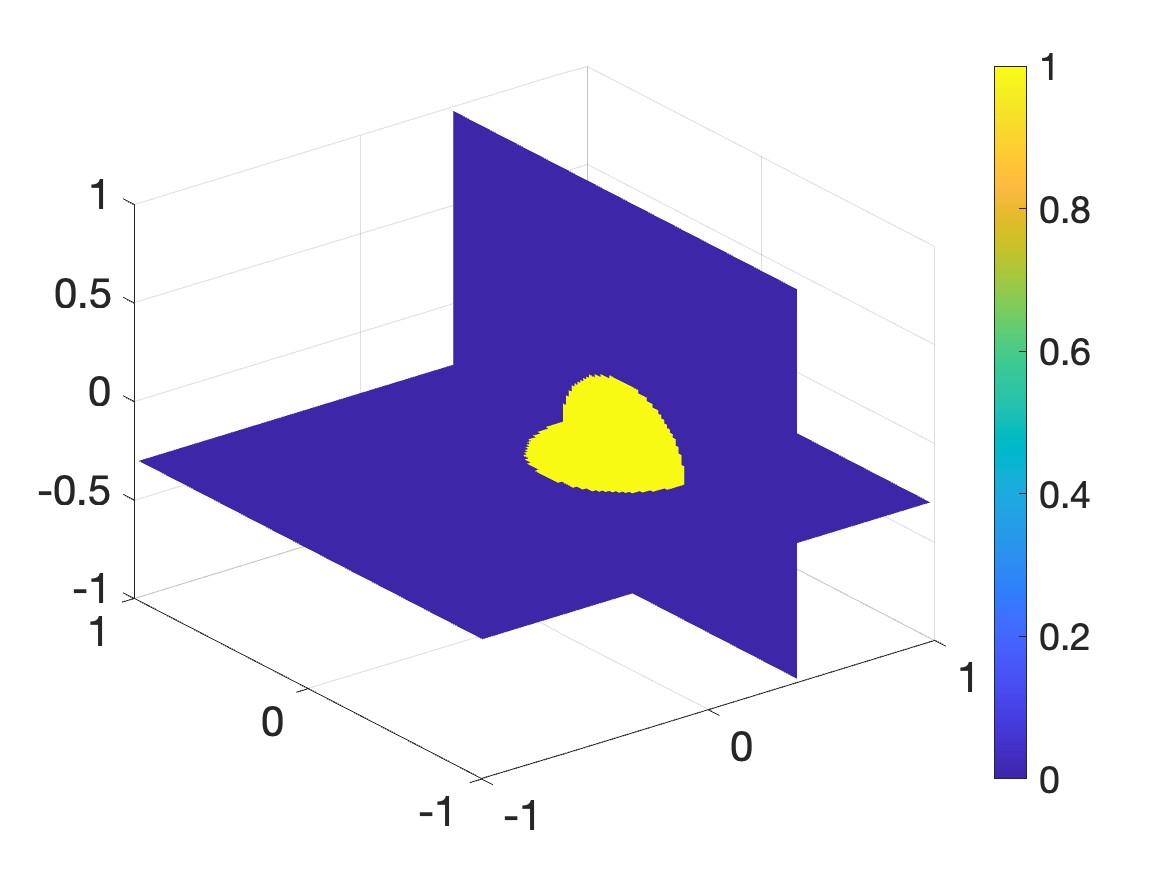}}
	\hfill
	\subfloat[]{\includegraphics[height = .2\textwidth,width = .3\textwidth]{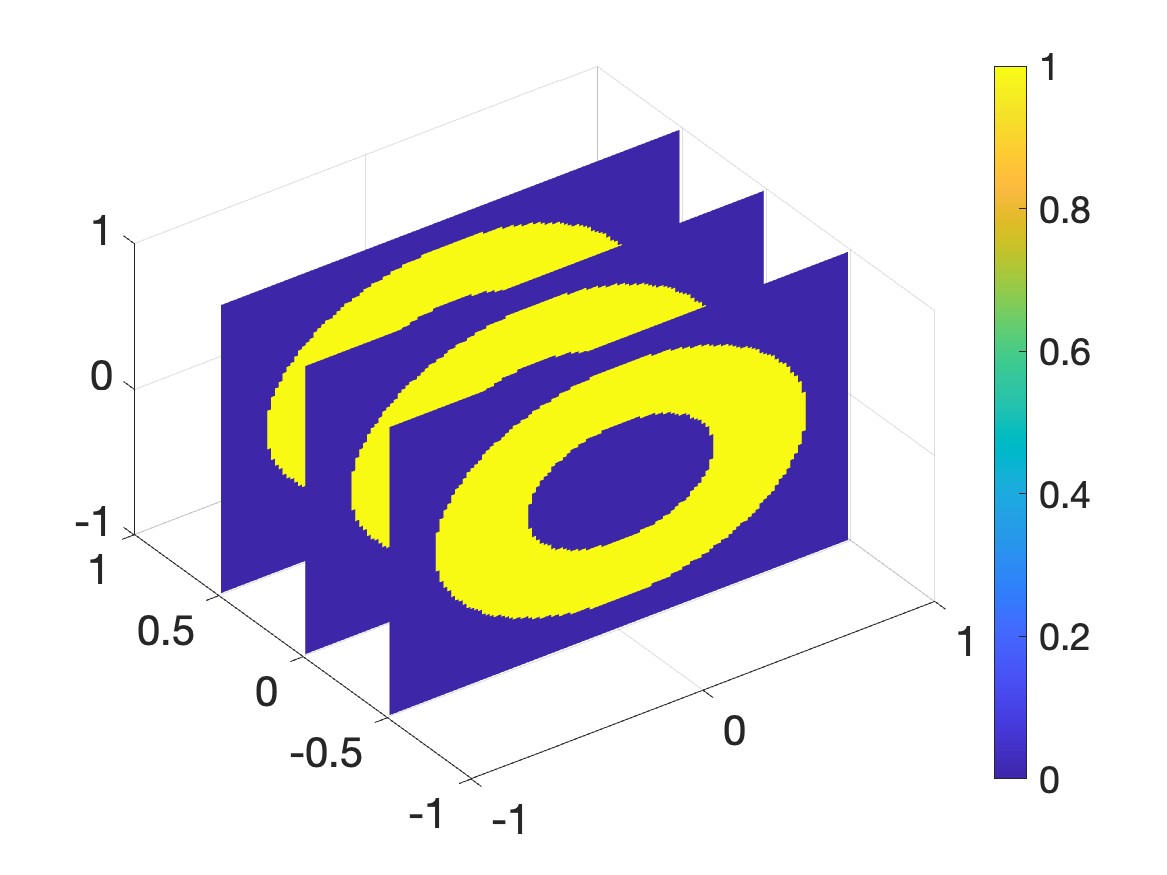}}
	\hfill
	\subfloat[]{\includegraphics[height = .2\textwidth,width = .3\textwidth]{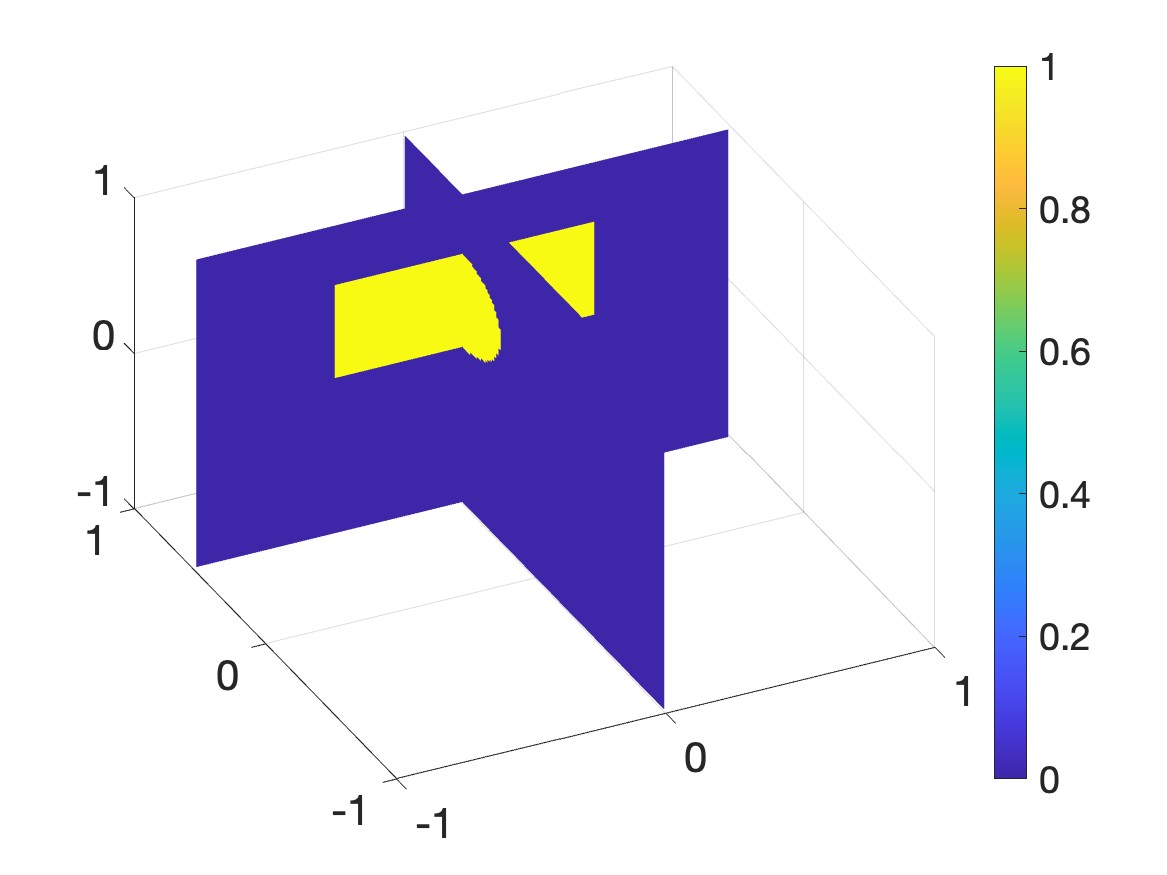}}
	
	\subfloat[]{\includegraphics[height = .2\textwidth,width = .3\textwidth]{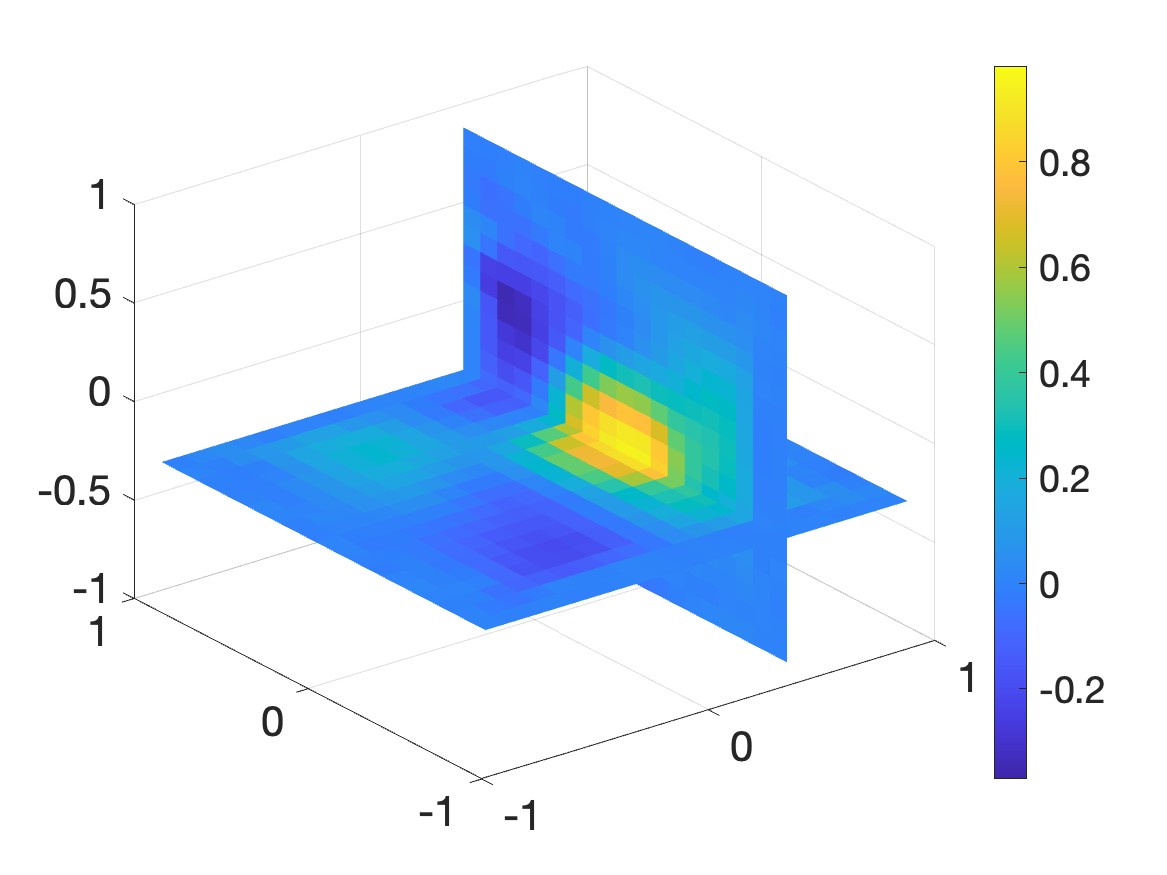}}
	\hfill
	\subfloat[]{\includegraphics[height = .2\textwidth,width = .3\textwidth]{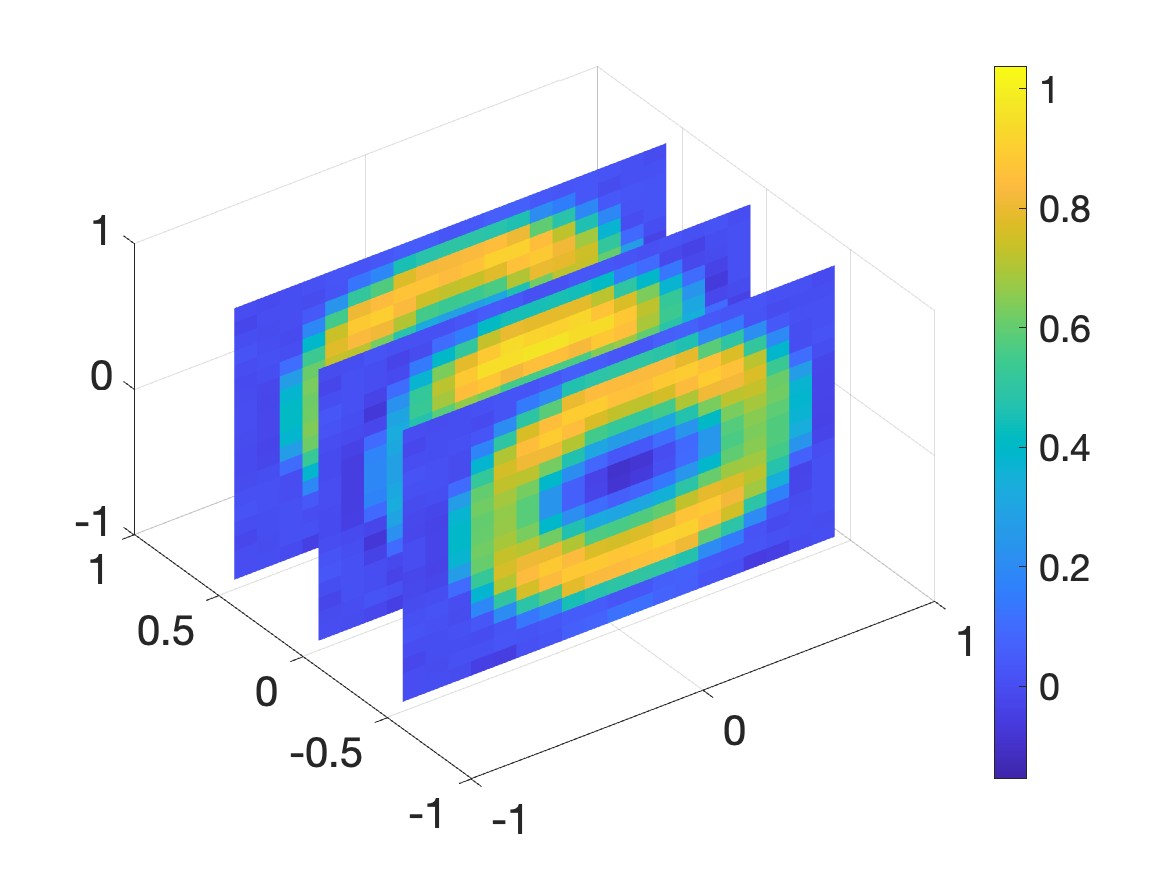}}
	\hfill
	\subfloat[]{\includegraphics[height = .2\textwidth,width = .3\textwidth]{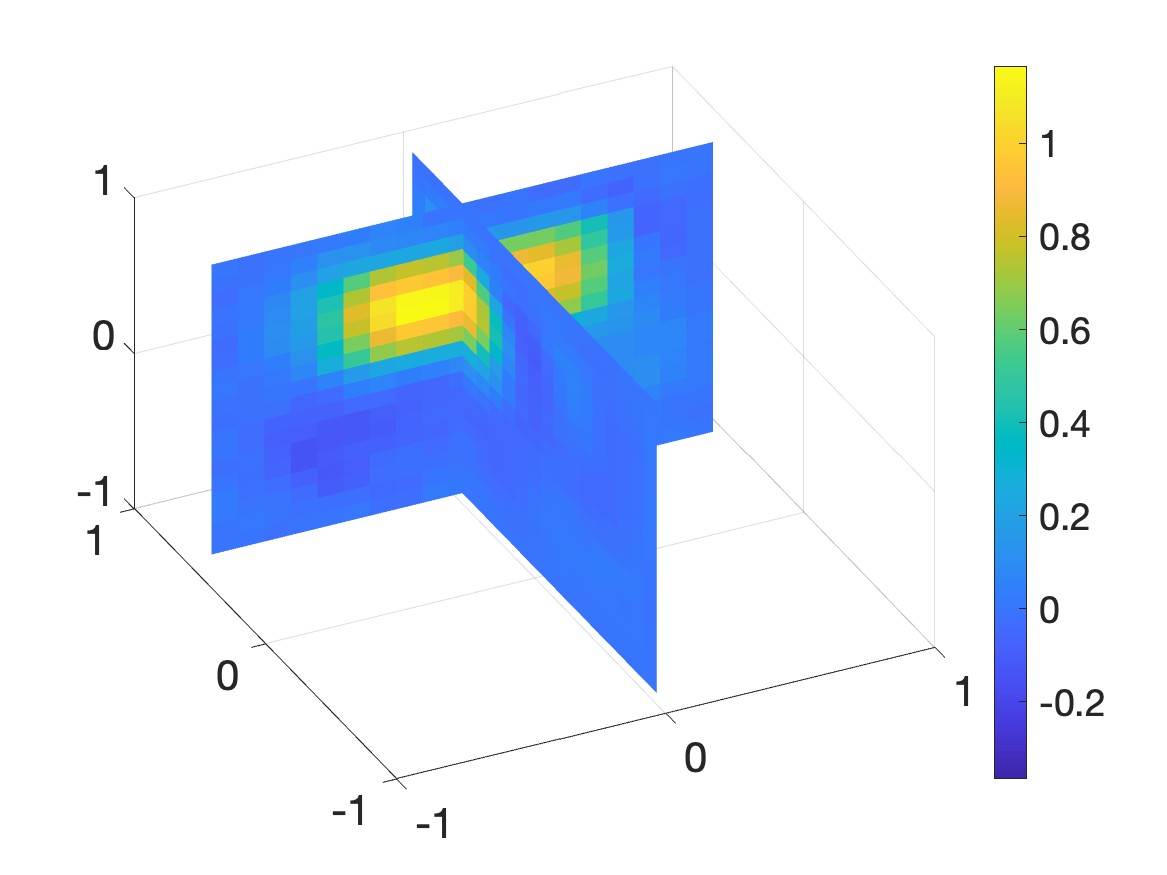}}
	
	\caption{
    Visualization of the true and reconstructed components of the initial electric field $\mathbf{E}^{\rm true}_0 = (E^{\rm true}_1, E^{\rm true}_2, E^{\rm true}_3)$. 
    Top row (a)--(c): isosurfaces of the true components. 
    Second row (d)--(f): isosurfaces of the reconstructed components $E^{\rm comp}_j$, $j = 1,2,3$. 
    Third row (g)--(i): representative 2D slices of the true components. 
    Bottom row (j)--(l): corresponding slices of the reconstructed components.  
    The figure demonstrates both geometric and intensity-level agreement between the true and reconstructed fields.
    }
    \label{fig_test1}
\end{figure}

Figure~\ref{fig_test1} provides a comprehensive comparison between the true and reconstructed initial field components using both isosurface and slice visualizations. The top two rows show that the reconstructed isosurfaces (d)--(f) accurately replicate the true geometric structures (a)--(c). In particular, the spherical inclusion in $E_1^{\rm true}$ is sharply recovered in $E_1^{\rm comp}$, while the cylindrical shell with a central hole in $E_2^{\rm true}$ is well-resolved in $E_2^{\rm comp}$, including the internal void. The small horizontally-oriented cylinder in $E_3^{\rm true}$ is also correctly located and shaped in the corresponding reconstruction.

The bottom two rows (g)--(l) further support the above observations by presenting 2D cross-sectional slices of the true and reconstructed fields. These visualizations confirm that the proposed method successfully recovers not only the spatial support but also the amplitude distribution of each component. The reconstructed field $E_1^{\rm comp}$ attains a maximum value of 0.9819, corresponding to a relative error of 1.8\%. For $E_2^{\rm comp}$, the peak value is 1.0377, yielding a relative error of 3.7\%. The maximum of $E_3^{\rm comp}$ is 1.1662, with a relative error of 16.62\%. These levels of discrepancy are acceptable given that the input data are contaminated with 10\% multiplicative noise. The reconstructed slices (j)--(l) exhibit smooth and well-centered intensity patterns that align closely with the true profiles (g)--(i), demonstrating both high spatial accuracy and effective contrast preservation. Collectively, these results highlight the robustness and precision of the reconstruction algorithm in recovering complex vector fields from noisy, indirect measurements.

\subsection*{Test 2}

We now proceed to Test 2 to further evaluate the performance of the proposed reconstruction method under a different configuration of the initial electric field. In this experiment, we consider a new set of geometries for $\mathbf{E}^{\rm true}_0$ designed to challenge the algorithm with increased structural complexity and spatial variation.
In this test, the true initial electric field $\mathbf{E}_0^{\rm true}(\mathbf{x}) = (E_1^{\rm true}, E_2^{\rm true}, E_3^{\rm true})$ consists of two spheres and two embedded letters in 3D space:
\[
E_1^{\rm true}(\mathbf{x}) =
\begin{cases}
1, & \text{if } (x + 0.55)^2 + y^2 + (z + 0.5)^2 < 0.3^2, \\
2, & \text{if } (x - 0.55)^2 + (y - 0.3)^2 + (z - 0.5)^2 < 0.3^2, \\
0, & \text{otherwise},
\end{cases}
\]
\[
E_2^{\rm true}(\mathbf{x}) = \chi_T(x, y, z),
\qquad
E_3^{\rm true}(\mathbf{x}) = \chi_Y(x, y, z),
\]
where $\chi_T$ and $\chi_Y$ are characteristic functions corresponding to the 3D extrusions of the 2D binary letter images 'T' and 'Y', respectively. The letter 'T' is extruded over the slab $z \in [-0.75, -0.3]$, and the letter 'Y' over $z \in [0.3, 0.9]$. 

\begin{figure}
    \centering
	\subfloat[]{\includegraphics[height = .2\textwidth,width = .3\textwidth]{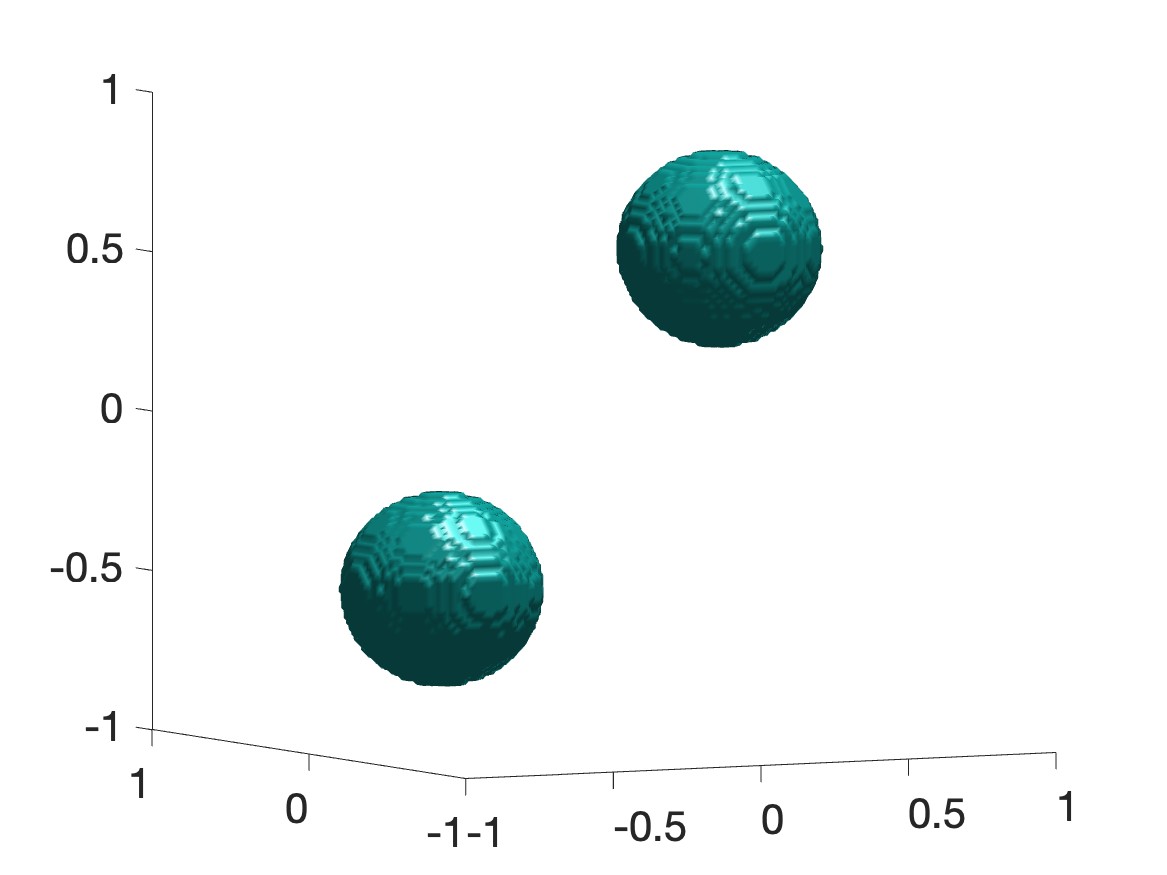}}
	\hfill
	\subfloat[]{\includegraphics[height = .2\textwidth,width = .3\textwidth]{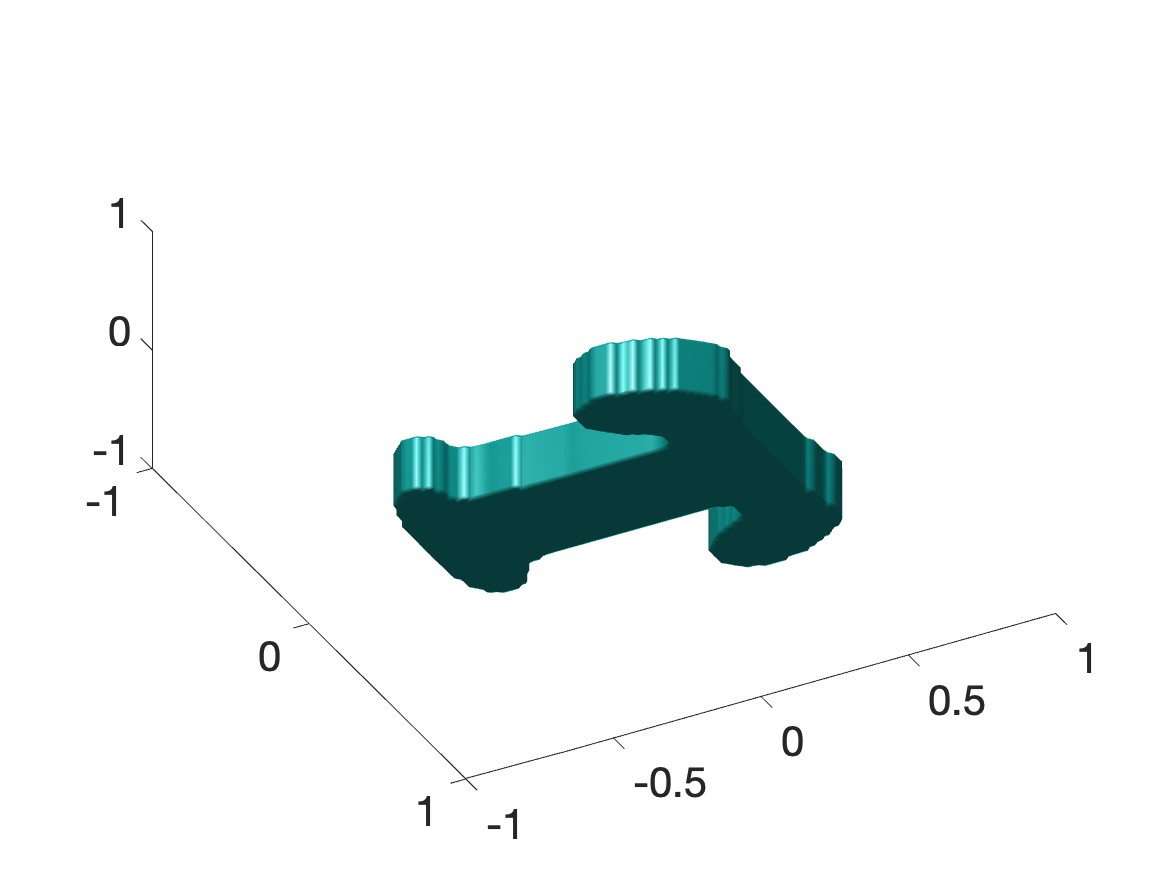}}
	\hfill
	\subfloat[]{\includegraphics[height = .2\textwidth,width = .3\textwidth]{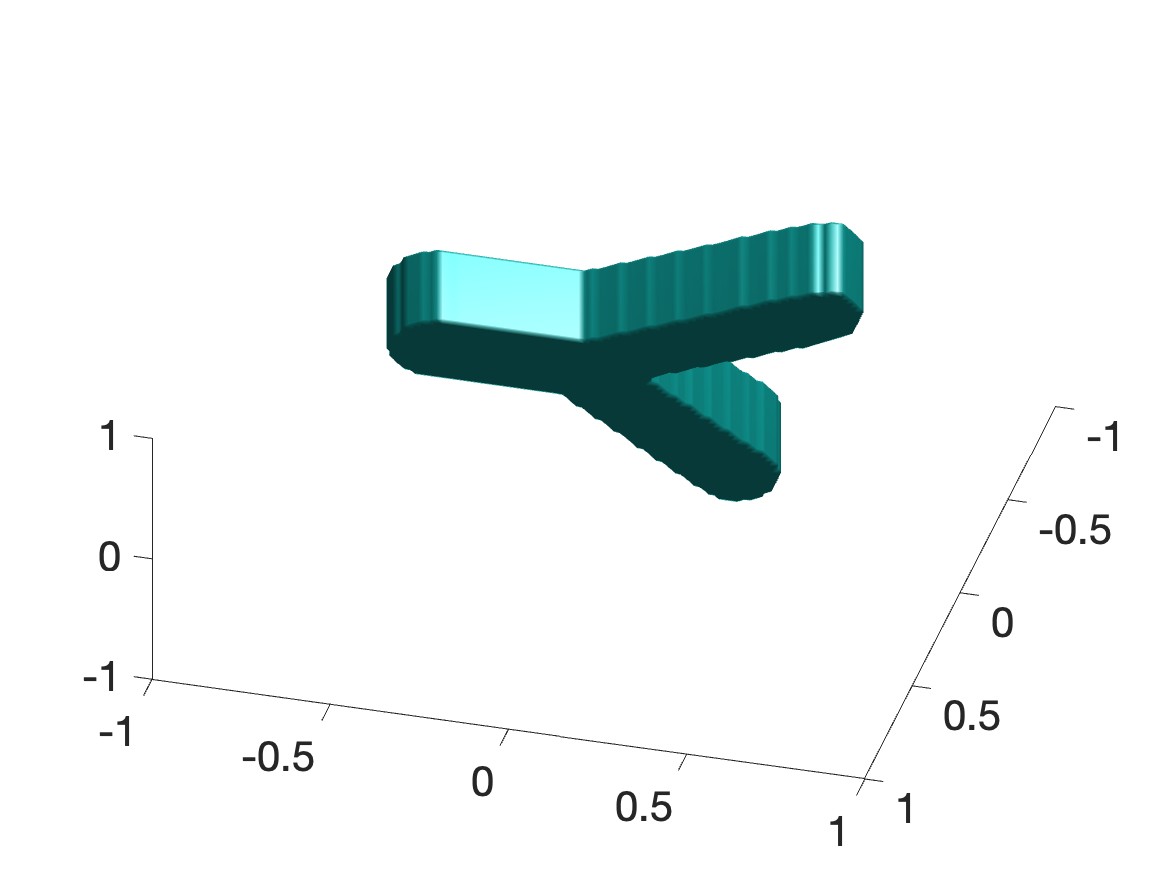}}
	
	\subfloat[]{\includegraphics[height = .2\textwidth,width = .3\textwidth]{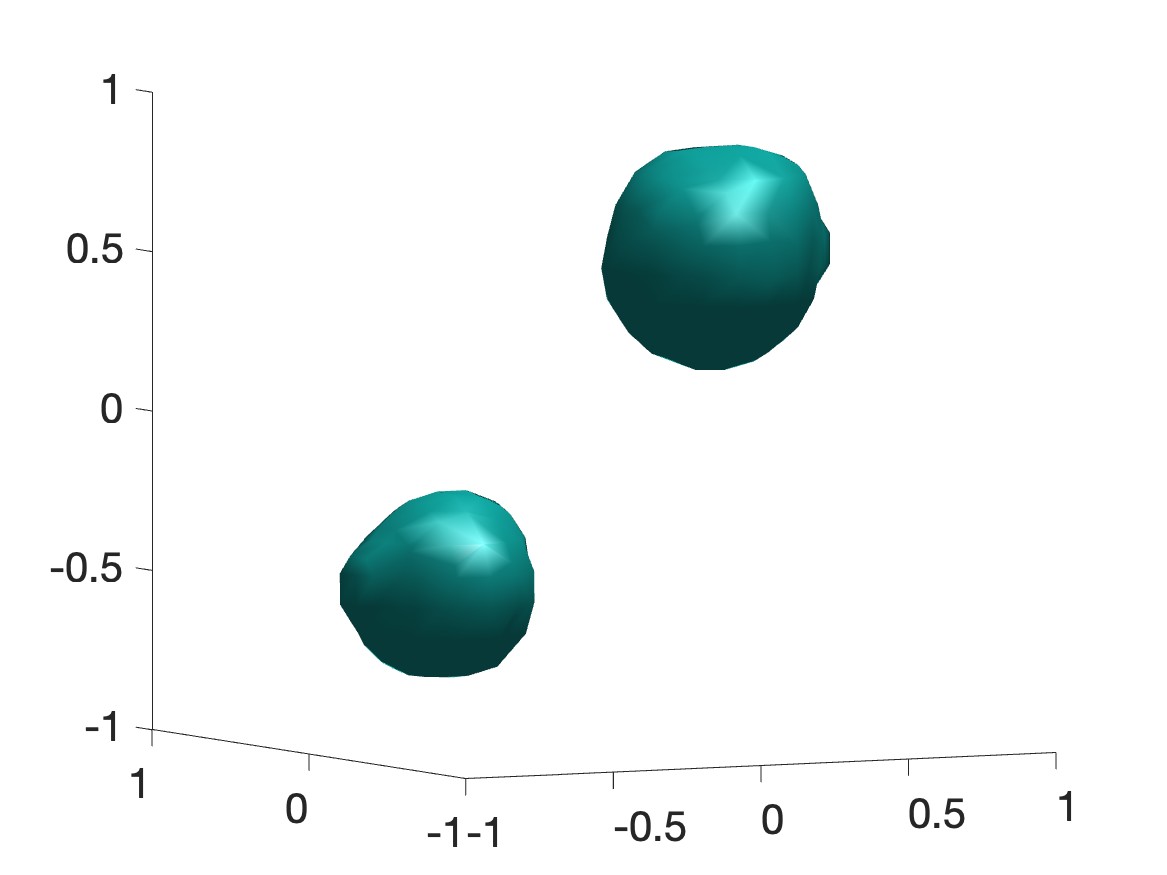}}
	\hfill
	\subfloat[]{\includegraphics[height = .2\textwidth,width = .3\textwidth]{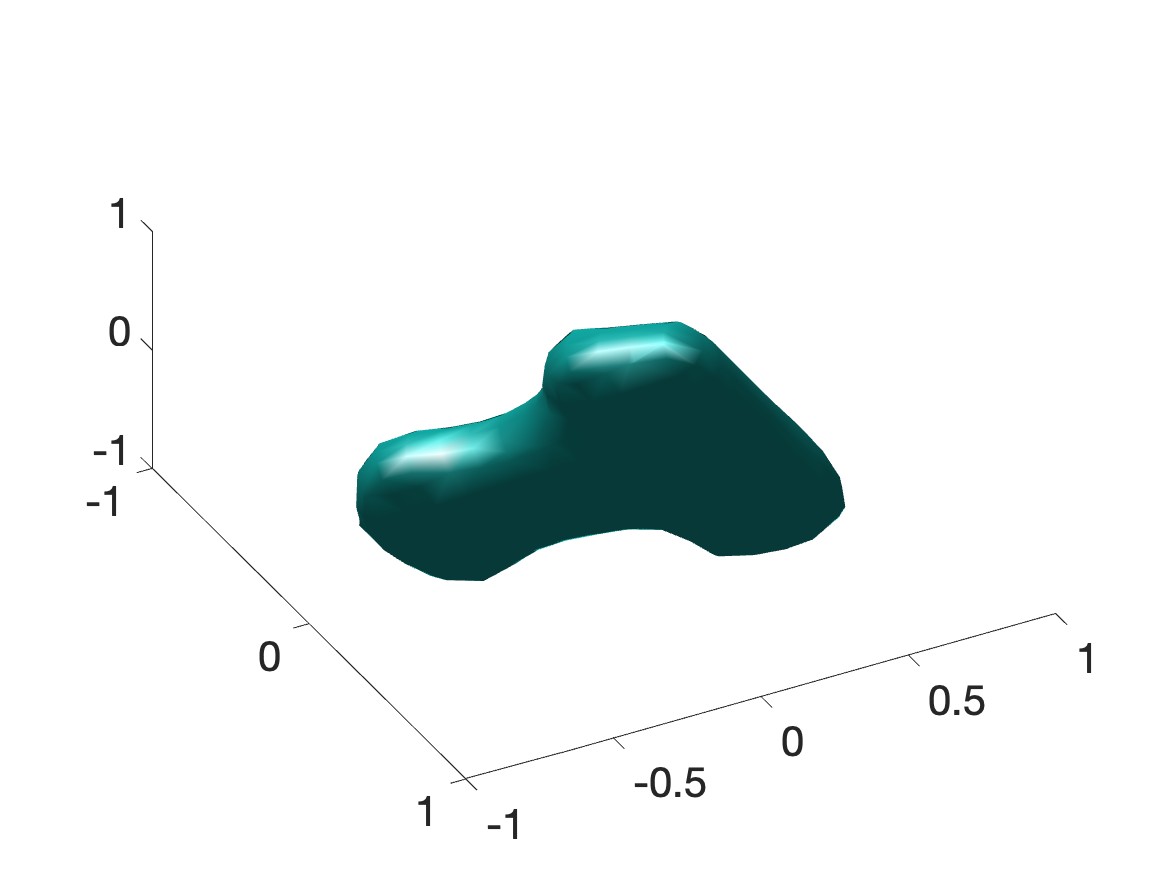}}
	\hfill
	\subfloat[]{\includegraphics[height = .2\textwidth,width = .3\textwidth]{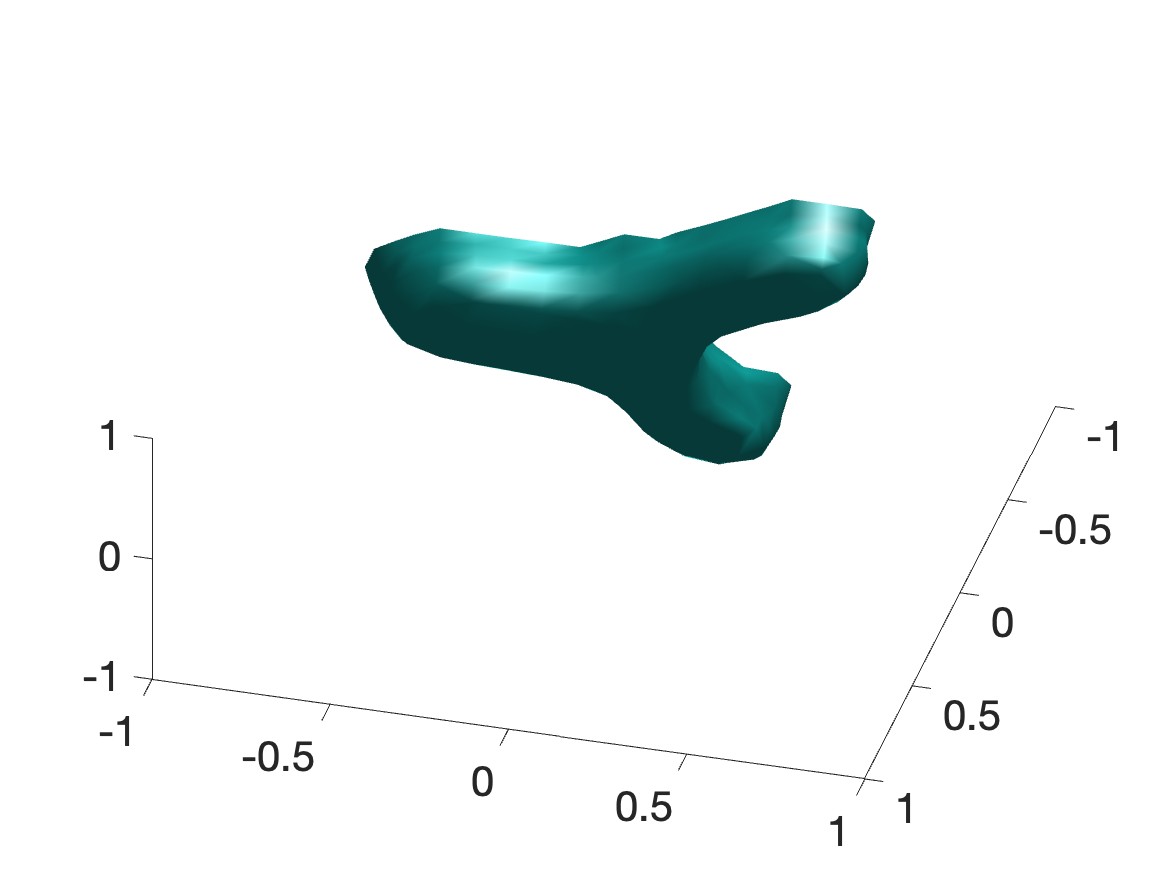}}
	
	\subfloat[]{\includegraphics[height = .2\textwidth,width = .3\textwidth]{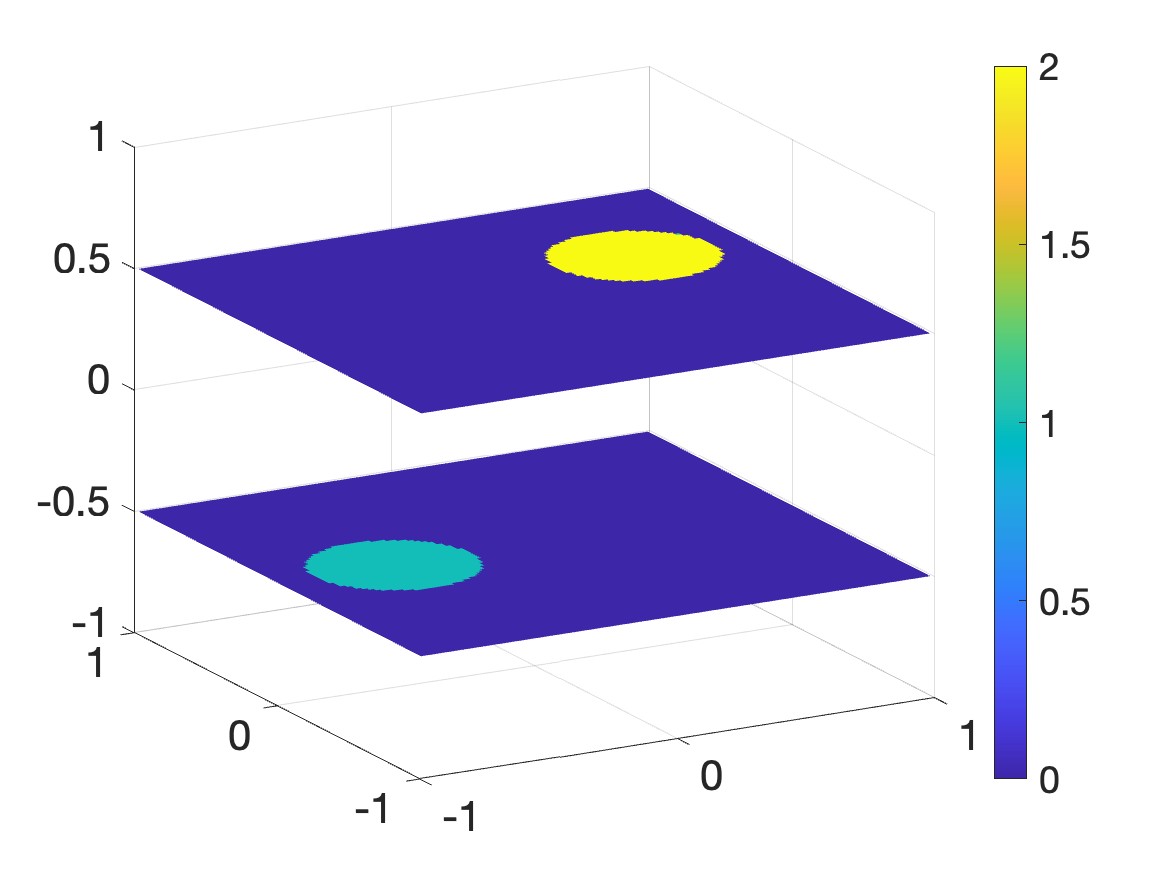}}
	\hfill
	\subfloat[]{\includegraphics[height = .2\textwidth,width = .3\textwidth]{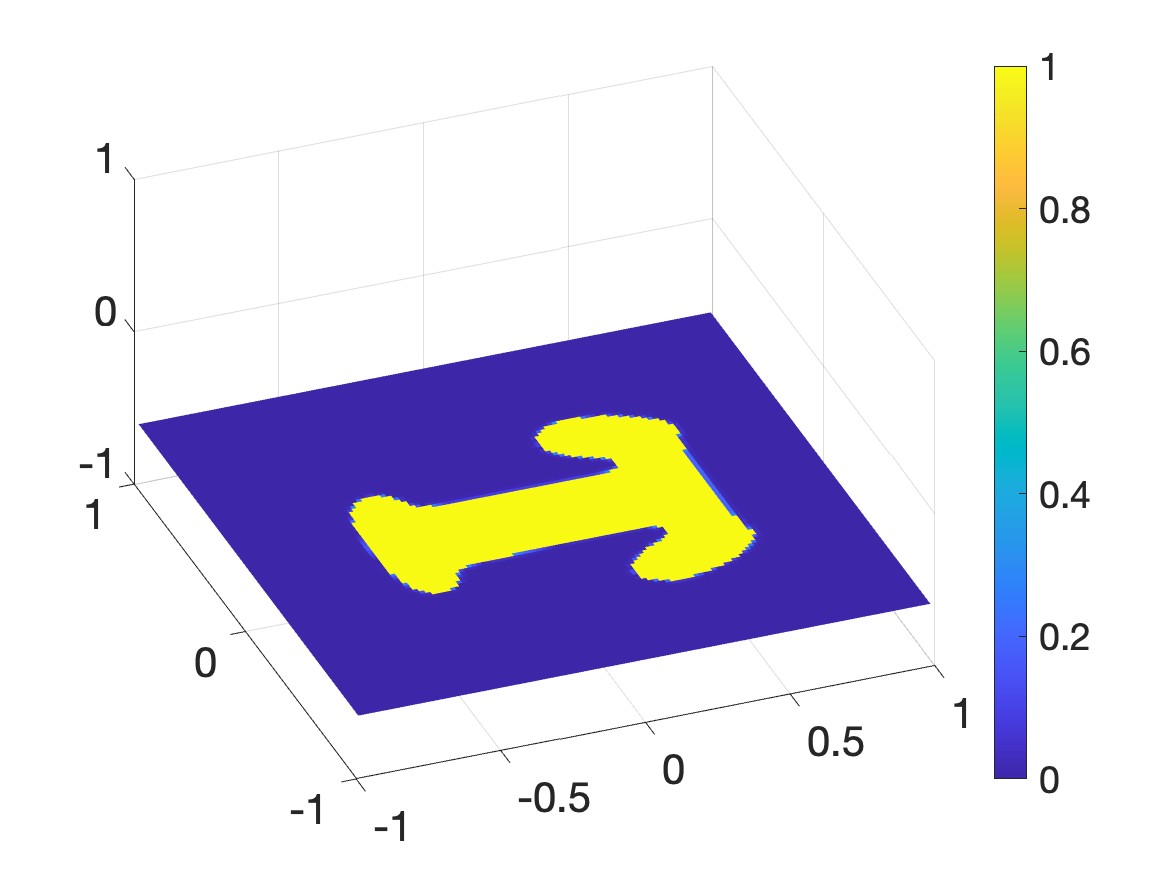}}
	\hfill
	\subfloat[]{\includegraphics[height = .2\textwidth,width = .3\textwidth]{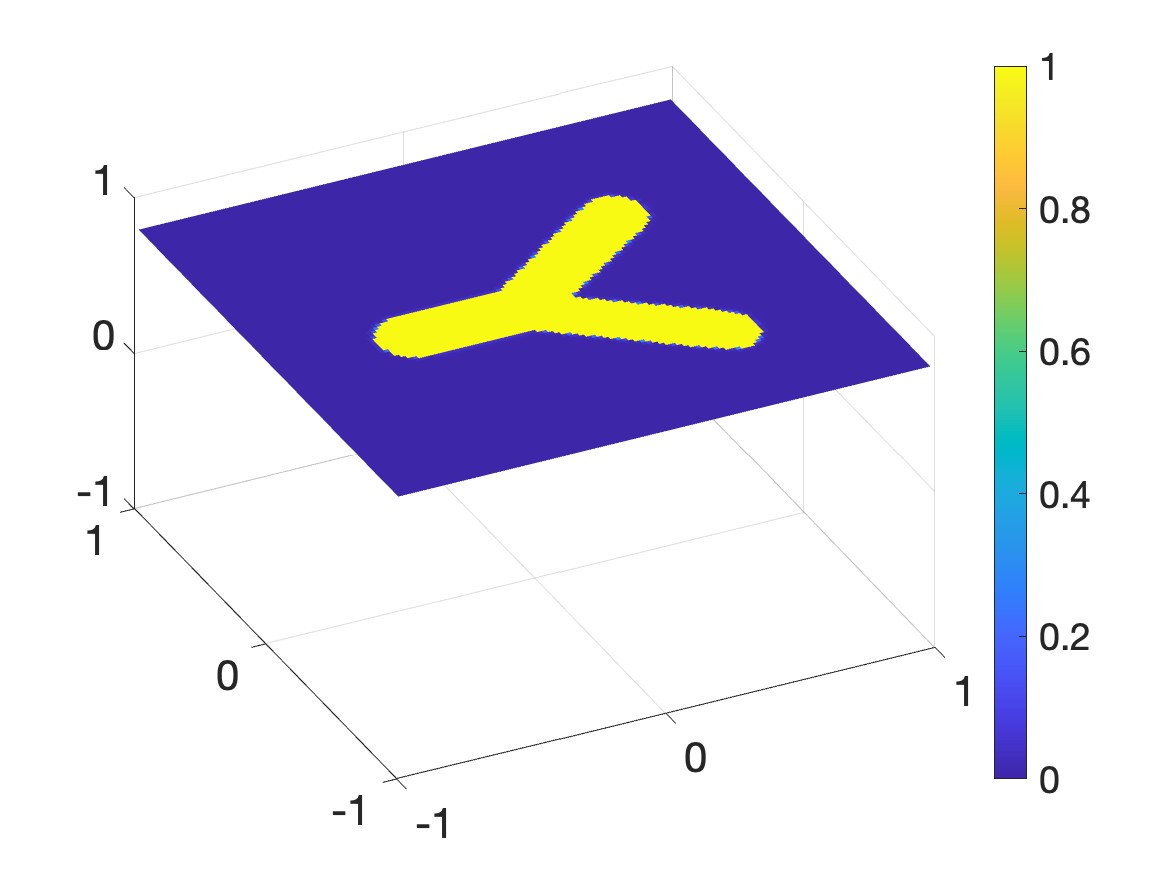}}
	
	\subfloat[]{\includegraphics[height = .2\textwidth,width = .3\textwidth]{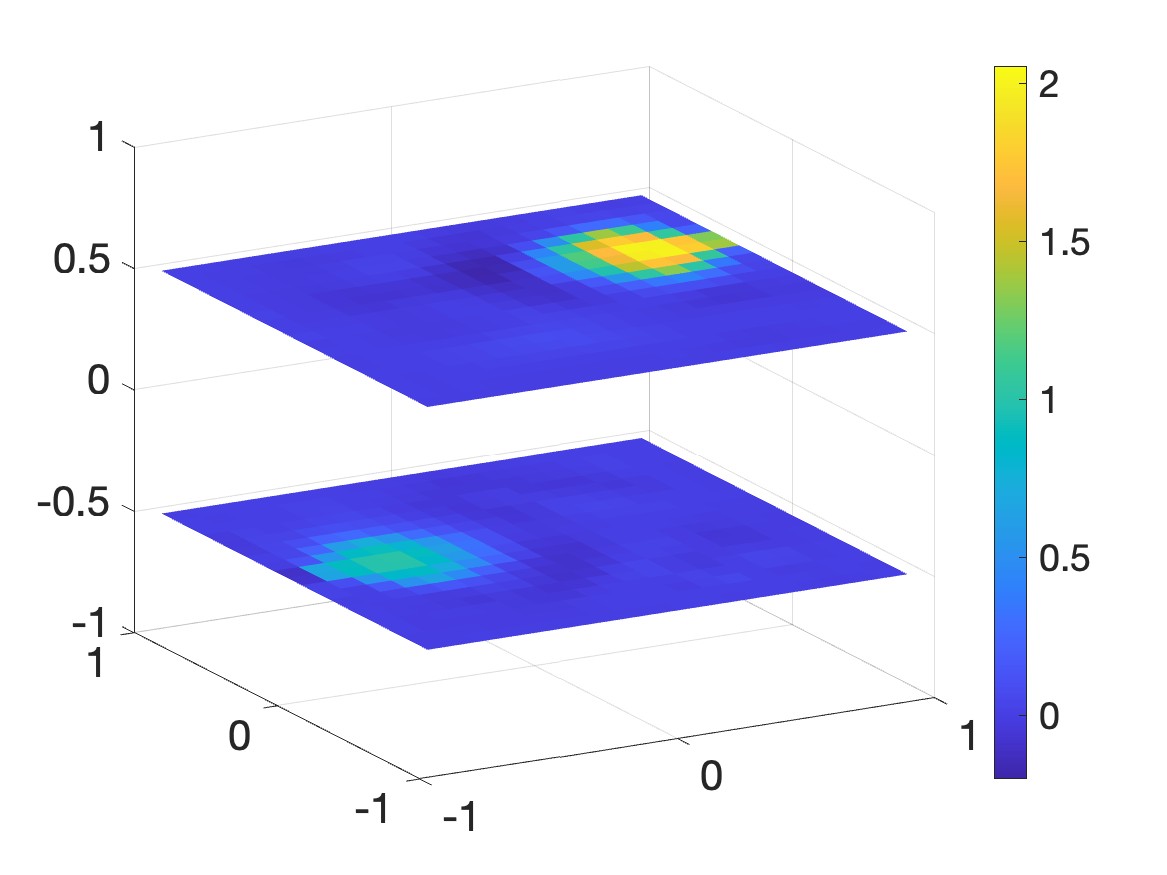}}
	\hfill
	\subfloat[]{\includegraphics[height = .2\textwidth,width = .3\textwidth]{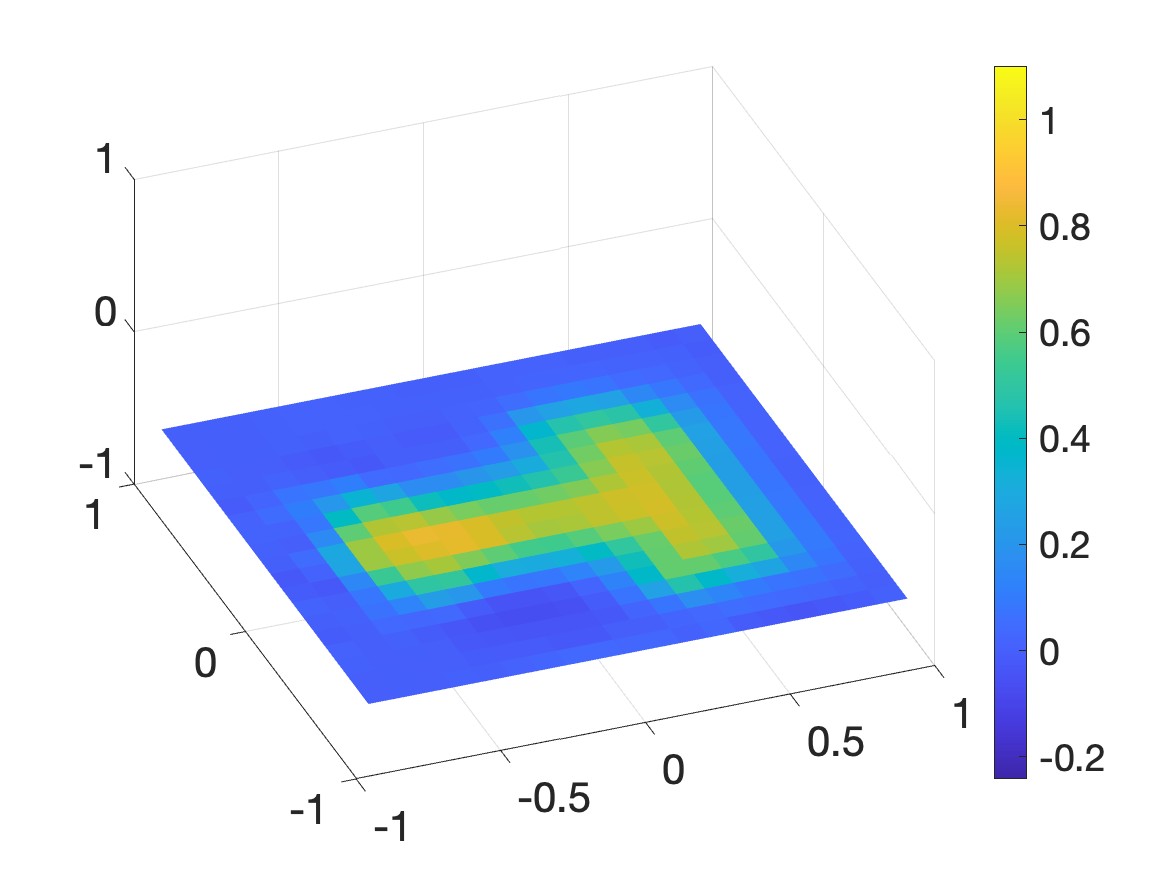}}
	\hfill
	\subfloat[]{\includegraphics[height = .2\textwidth,width = .3\textwidth]{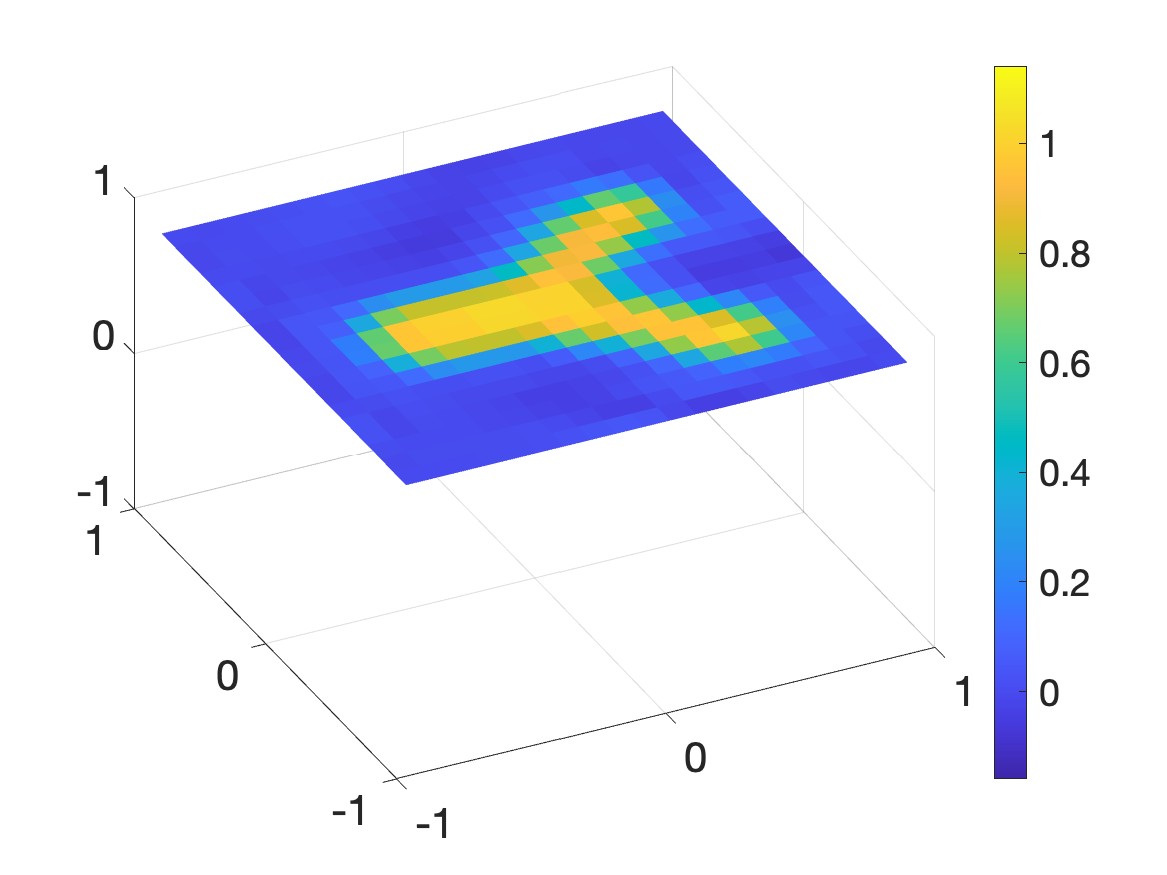}}
	
	\caption{
    Visualization of the true and reconstructed components of the initial electric field $\mathbf{E}^{\rm true}_0 = (E^{\rm true}_1, E^{\rm true}_2, E^{\rm true}_3)$. 
    Top row (a)--(c): isosurfaces of the true components. 
    Second row (d)--(f): isosurfaces of the reconstructed components $E^{\rm comp}_j$, $j = 1,2,3$. 
    Third row (g)--(i): representative 2D slices of the true components. 
    Bottom row (j)--(l): corresponding slices of the reconstructed components.  
    The figure demonstrates both geometric and intensity-level agreement between the true and reconstructed fields.
    }
    \label{fig_test2}
\end{figure}

Figure~\ref{fig_test2} illustrates the reconstruction results for Test 2 under 10\% multiplicative noise. The first component $E^{\rm true}_1$ contains two distinct spherical inclusions of different amplitudes, both of which are accurately captured in the reconstruction $E^{\rm comp}_1$ (subfigures a and d), with correct relative sizing and spatial separation. The second and third components embed the letters 'T' and 'Y' as volumetric shapes. The isosurfaces of $E^{\rm comp}_2$ and $E^{\rm comp}_3$ (subfigures e and f) show that the reconstruction effectively preserves the topology and orientation of the original characters (subfigures b and c), despite the presence of significant noise.

The 2D slices in subfigures (g)--(l) offer further insight into the reconstruction fidelity. The slices of $E^{\rm comp}_1$ (j) accurately reflect the location and amplitude of both spherical regions observed in the true field (g). For the symbolic components, $E^{\rm comp}_2$ (k) and $E^{\rm comp}_3$ (l) preserve the structure of the letters 'T' and 'Y' with only minor diffusion and smooth deformation, demonstrating resilience to noise and strong contrast recovery. Overall, the reconstructions exhibit excellent agreement with the ground truth in both geometry and intensity, confirming the robustness of the method under noisy conditions.

Quantitatively, the reconstructed upper sphere in $E_1^{\rm comp}$ reaches a maximum value of 2.0553, corresponding to a relative error of 2.76\%, while the lower sphere attains a maximum of 1.0395, yielding a relative error of 3.95\%. For the symbolic components, the peak value within the letter 'T' in $E_2^{\rm comp}$ is 1.1014, with a relative error of 10.14\%, and within the letter 'Y' in $E_3^{\rm comp}$ the maximum is 1.1411, resulting in a relative error of 14.11\%.

\subsection*{Test 3}

In this test, the true initial electric field $\mathbf{E}_0^{\rm true}(\mathbf{x}) = (E_1^{\rm true}, E_2^{\rm true}, E_3^{\rm true})$ is defined by the following component-wise expressions:
\[
E_1^{\rm true}(\mathbf{x}) =
\begin{cases}
2.5, & \text{if } \max\{5|x + 0.55|, |y|\} < 0.9 \text{ and } |z + 0.4| < 0.3, \\
3,   & \text{if } (x - 0.55)^2 + y^2 + (z - 0.4)^2 < 0.3^2, \\
0,   & \text{otherwise},
\end{cases}
\]
\[
E_2^{\rm true}(\mathbf{x}) =
\begin{cases}
2.5, & \text{if } \max\{5|x + 0.5|, |y|\} < 0.9 \text{ and } |z + 0.4| < 0.3, \\
3,   & \text{if } \max\{5|x - 0.5|, |z|\} < 0.9 \text{ and } |y - 0.5| < 0.3, \\
0,   & \text{otherwise},
\end{cases}
\]
\[
E_3^{\rm true}(\mathbf{x}) =
\begin{cases}
2, & \text{if } (x - 0.5)^2 + (y - 0.4)^2 + (z - 0.3)^2 < 0.3^2, \\
0, & \text{otherwise}.
\end{cases}
\]

\begin{figure}
    \centering
	\subfloat[]{\includegraphics[height = .2\textwidth,width = .3\textwidth]{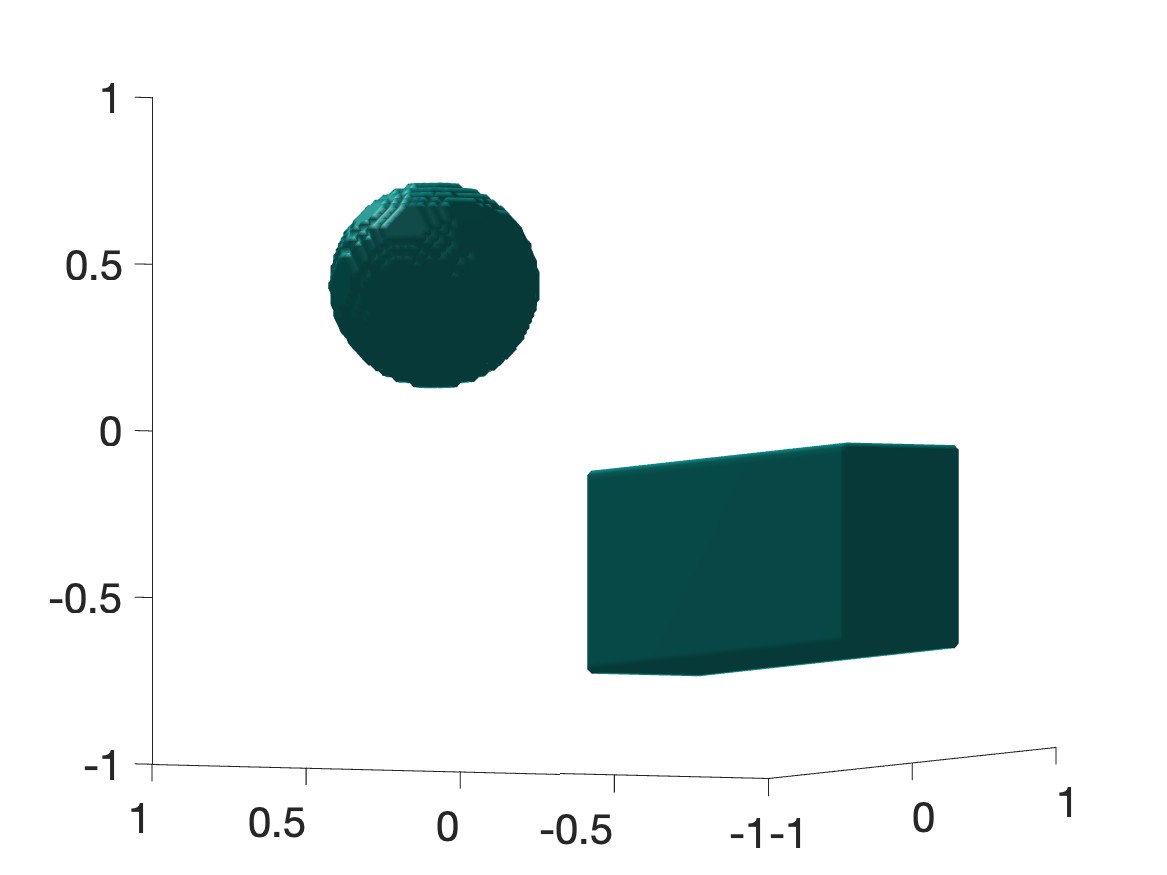}}
	\hfill
	\subfloat[]{\includegraphics[height = .2\textwidth,width = .3\textwidth]{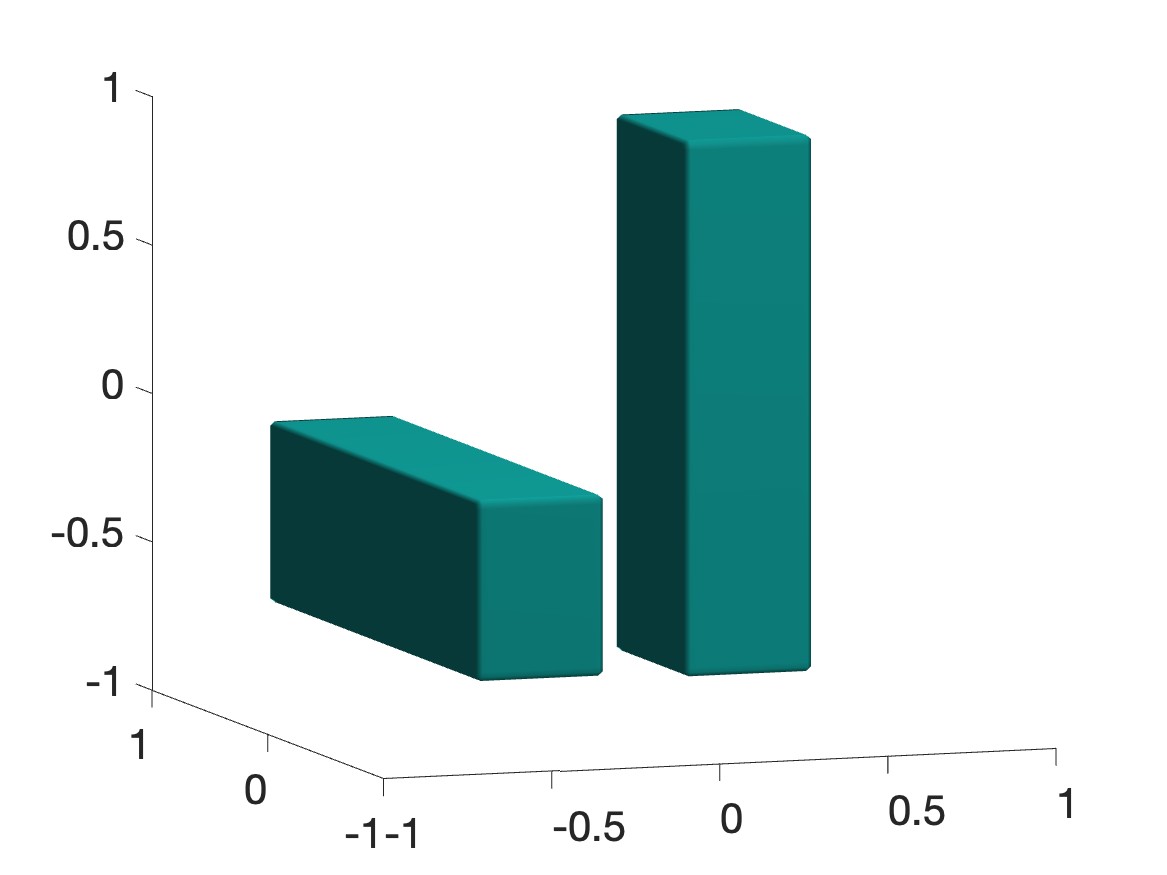}}
	\hfill
	\subfloat[]{\includegraphics[height = .2\textwidth,width = .3\textwidth]{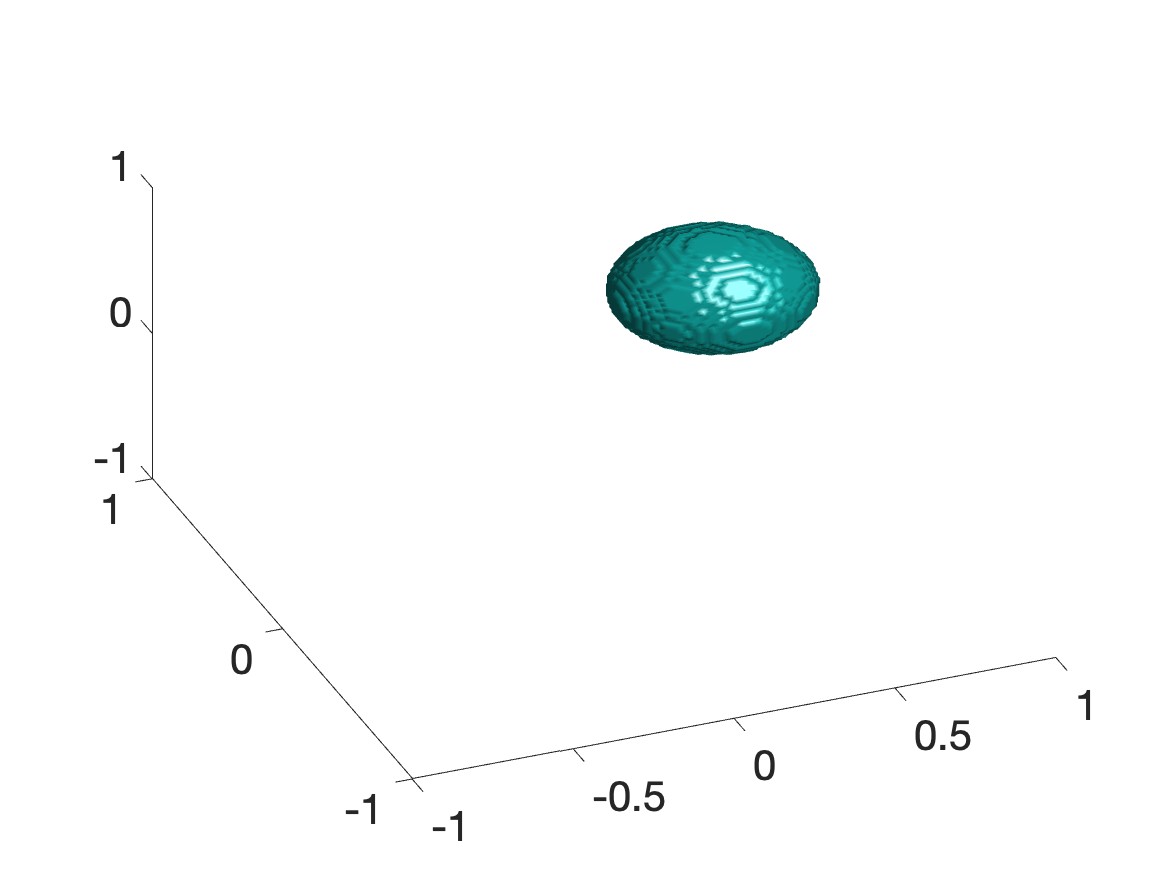}}
	
	\subfloat[]{\includegraphics[height = .2\textwidth,width = .3\textwidth]{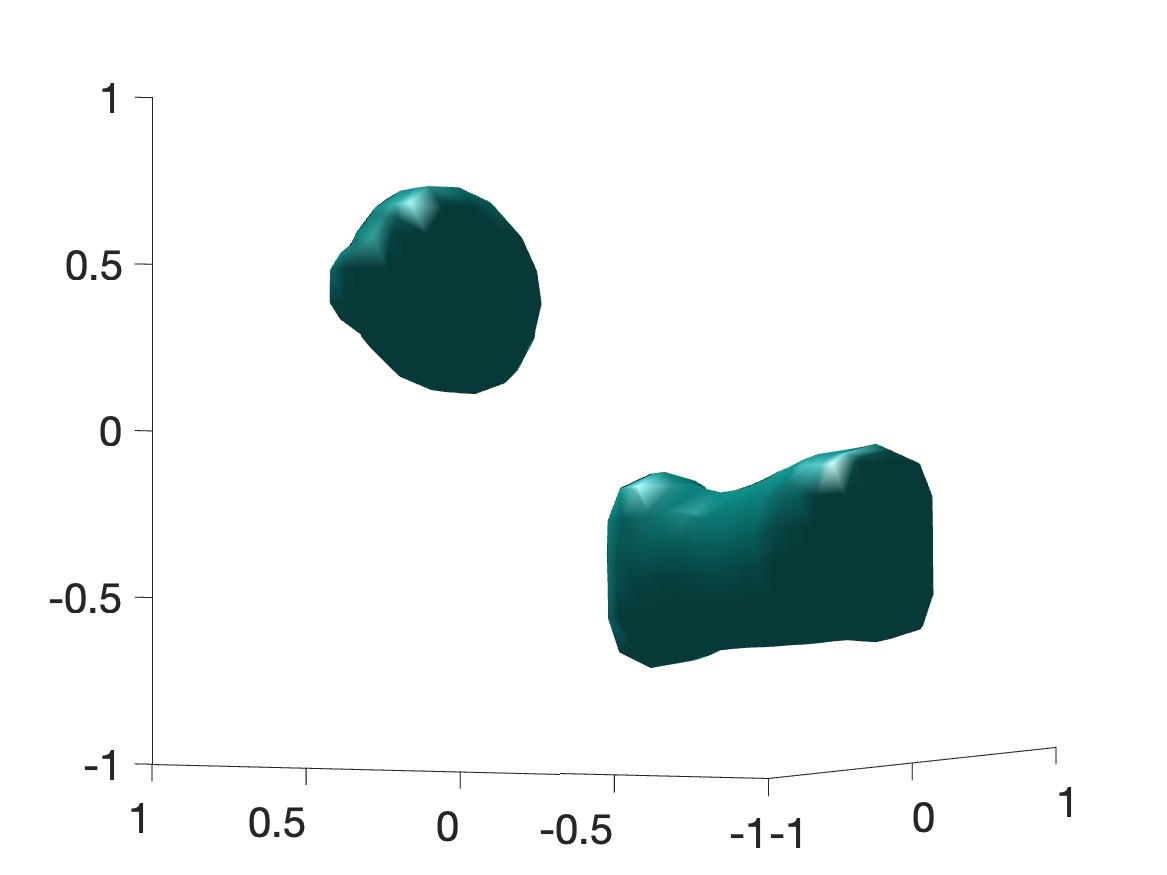}}
	\hfill
	\subfloat[]{\includegraphics[height = .2\textwidth,width = .3\textwidth]{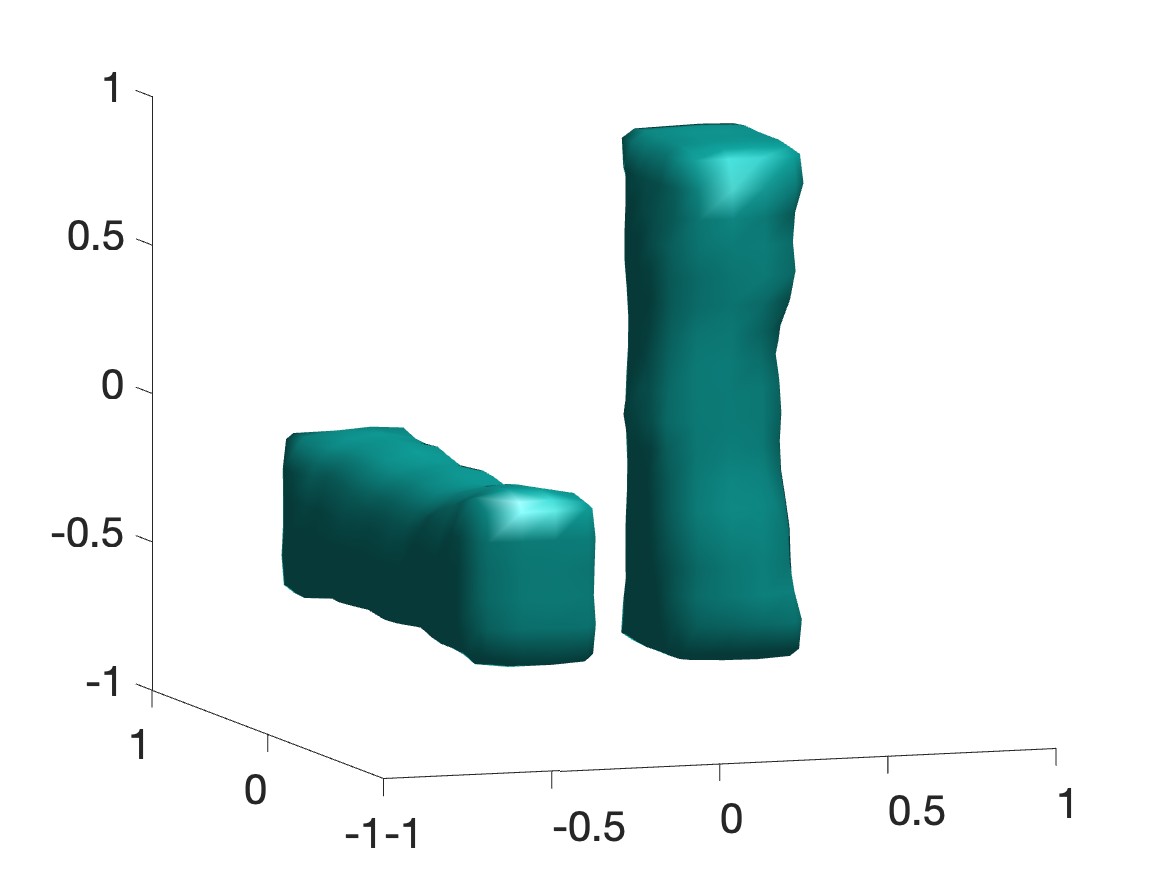}}
	\hfill
	\subfloat[]{\includegraphics[height = .2\textwidth,width = .3\textwidth]{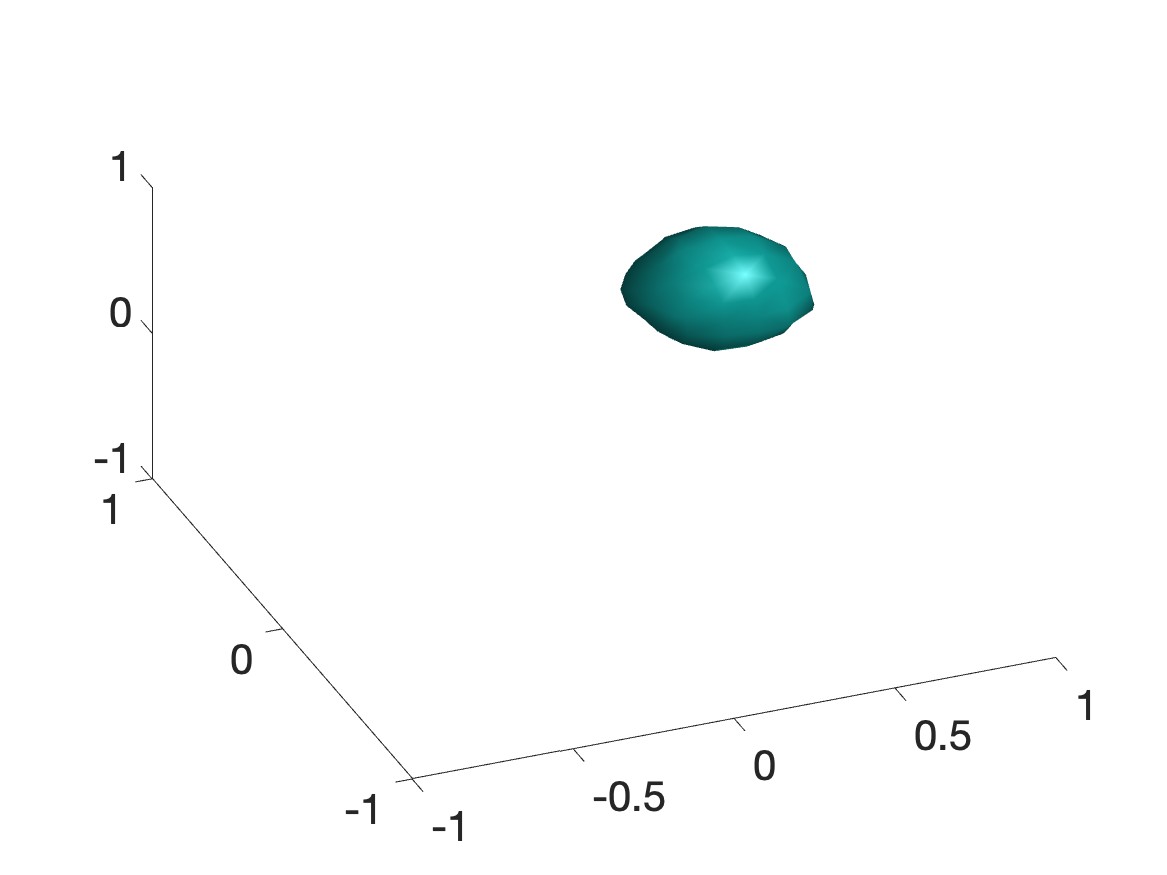}}
	
	\subfloat[]{\includegraphics[height = .2\textwidth,width = .3\textwidth]{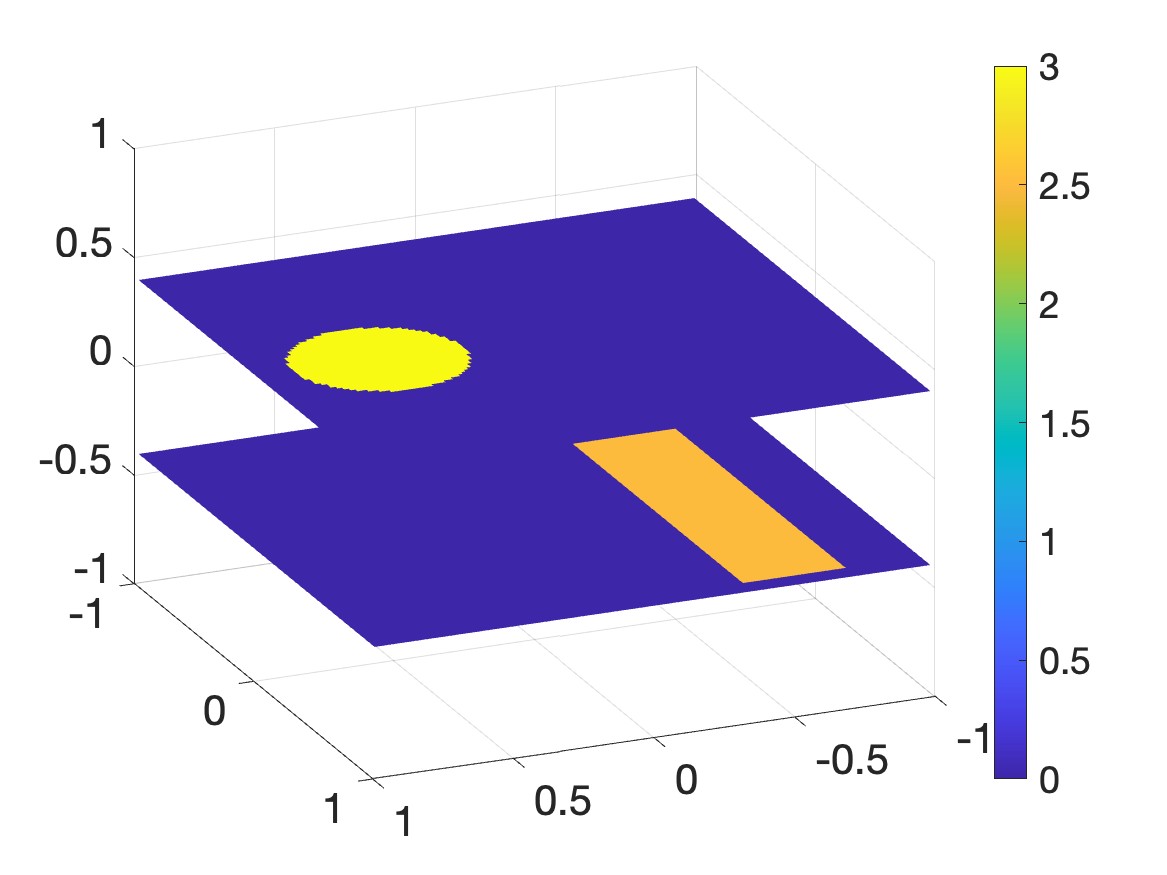}}
	\hfill
	\subfloat[]{\includegraphics[height = .2\textwidth,width = .3\textwidth]{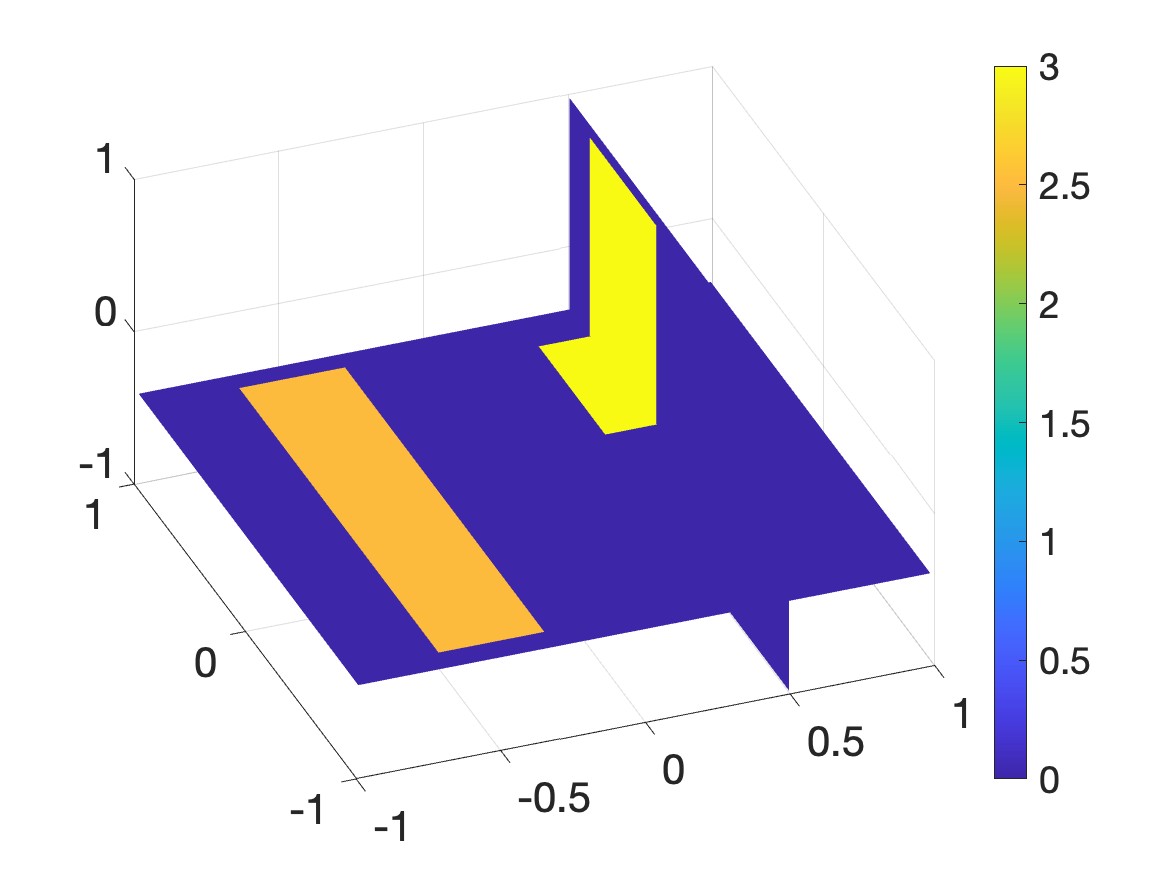}}
	\hfill
	\subfloat[]{\includegraphics[height = .2\textwidth,width = .3\textwidth]{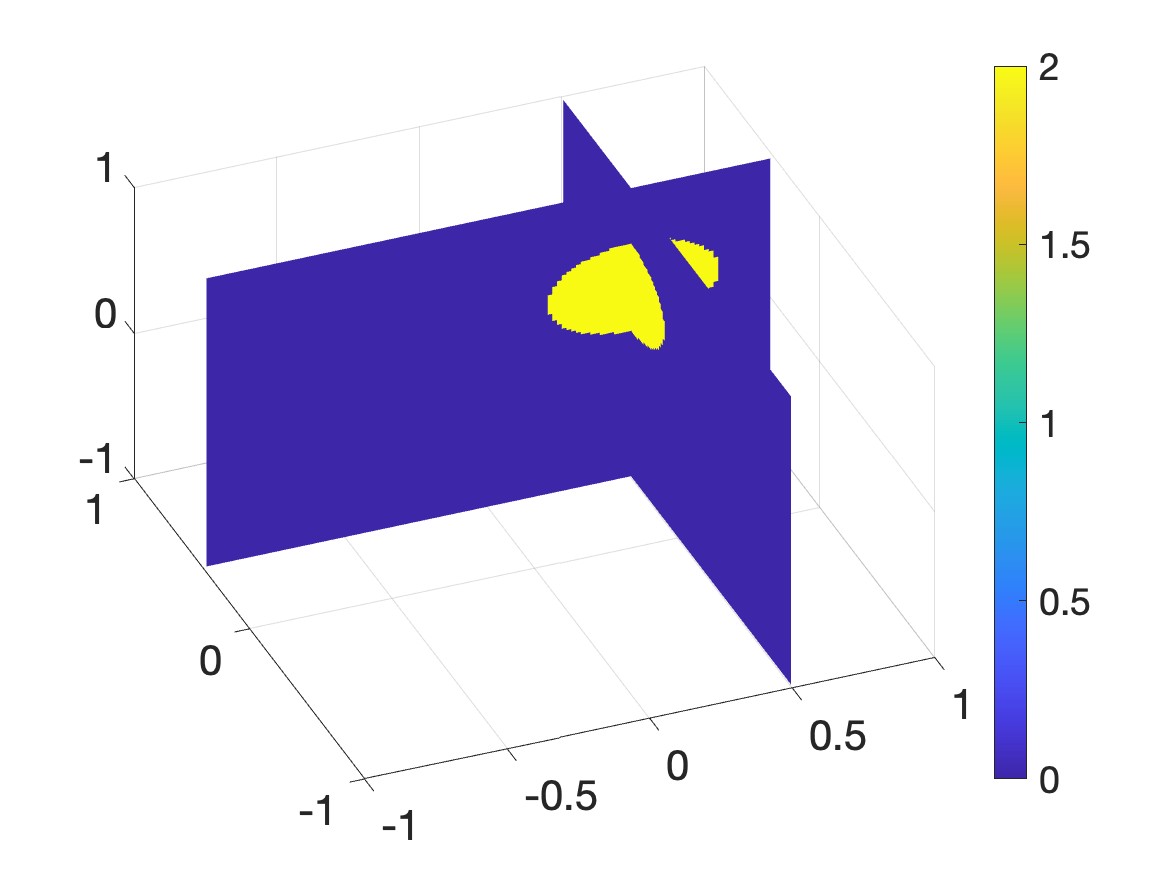}}
	
	\subfloat[]{\includegraphics[height = .2\textwidth,width = .3\textwidth]{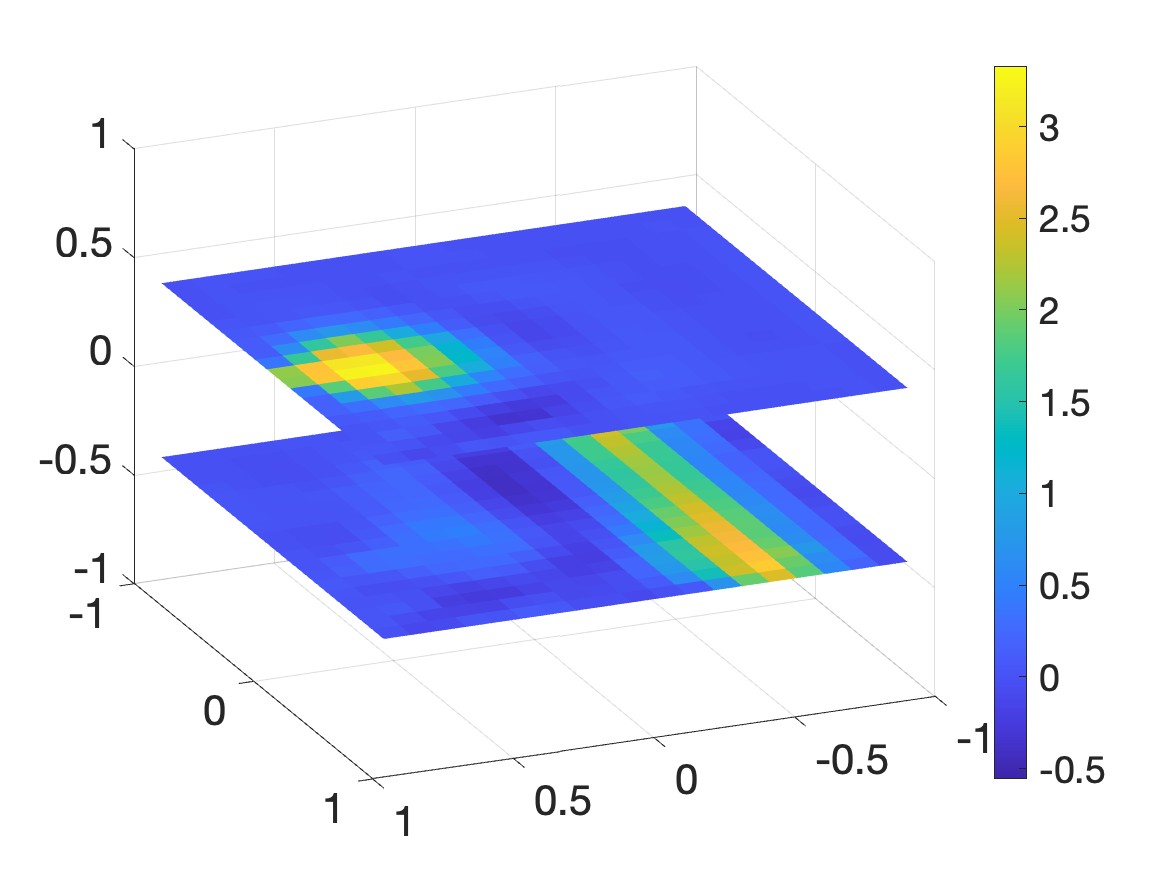}}
	\hfill
	\subfloat[]{\includegraphics[height = .2\textwidth,width = .3\textwidth]{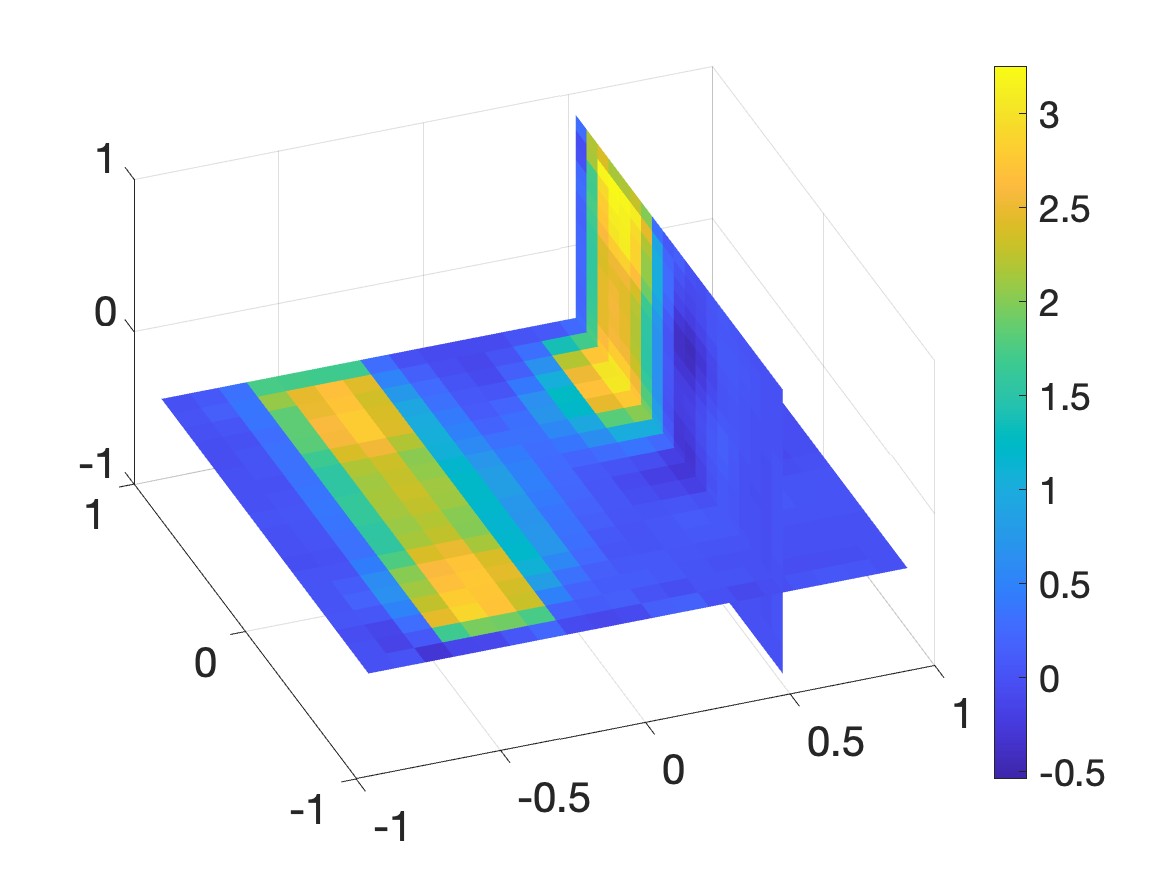}}
	\hfill
	\subfloat[]{\includegraphics[height = .2\textwidth,width = .3\textwidth]{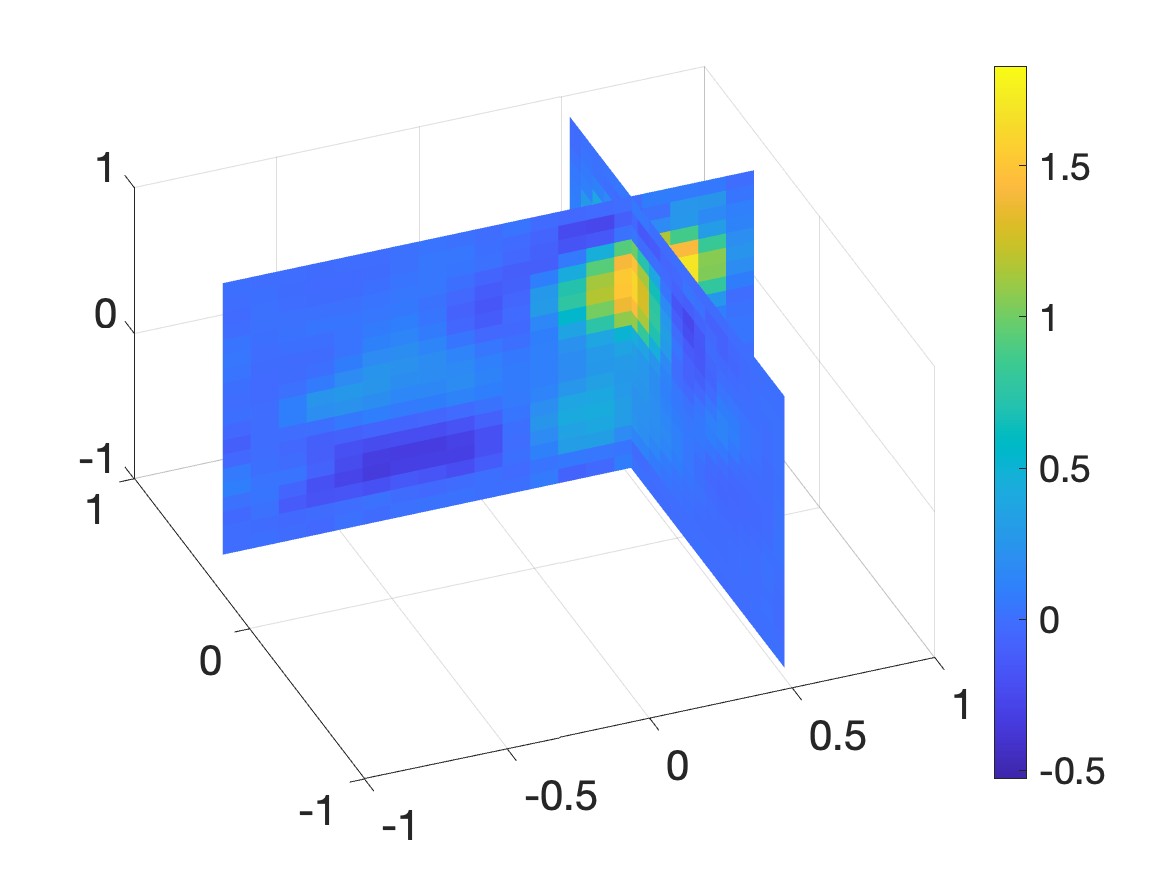}}
	
	\caption{
    Visualization of the true and reconstructed components of the initial electric field $\mathbf{E}^{\rm true}_0 = (E^{\rm true}_1, E^{\rm true}_2, E^{\rm true}_3)$. 
    Top row (a)--(c): isosurfaces of the true components. 
    Second row (d)--(f): isosurfaces of the reconstructed components $E^{\rm comp}_j$, $j = 1,2,3$. 
    Third row (g)--(i): representative 2D slices of the true components. 
    Bottom row (j)--(l): corresponding slices of the reconstructed components.  
    The figure demonstrates both geometric and intensity-level agreement between the true and reconstructed fields.
    }
    \label{fig_test3}
\end{figure}

Figure~\ref{fig_test3} presents the reconstruction results for Test 3, which features a mix of anisotropic slab-like structures and spherical inclusions with varying amplitudes. The top two rows show the isosurface plots of the true (a)--(c) and reconstructed (d)--(f) components. In $E_1^{\rm true}$, the algorithm successfully recovers both the spherical inclusion on the right and the elongated rectangular slab on the left. The reconstruction $E_1^{\rm comp}$ (subfigure d) maintains the separation and relative amplitudes of these two regions, with smooth transitions and accurate geometric profiles.

In $E_2^{\rm true}$ (subfigure b), the geometry includes two orthogonally oriented slab-like inclusions. The reconstruction $E_2^{\rm comp}$ (e) captures both features with reasonable sharpness and intensity contrast, though minor deformation is observed near the corners due to the anisotropic structure. The spherical region in $E_3^{\rm true}$ (c) is also accurately reconstructed in $E_3^{\rm comp}$ (f), with the location and shape well preserved.

The bottom two rows (g)--(l) show 2D cross-sectional slices of the true and reconstructed components. The slices of $E_1^{\rm comp}$ (j) and $E_2^{\rm comp}$ (k) demonstrate that the spatial support and amplitudes of the rectangular and spherical regions are well maintained, although slight blurring occurs at sharp edges. In $E_3^{\rm comp}$ (l), the intensity distribution closely matches that of the true slice (i), confirming accurate localization and amplitude reconstruction.

Quantitatively, the reconstructed field $\mathbf{E}^{\rm comp}$ yields the following maximum values in their respective target regions. Inside the disk centered at $(0.55, 0, 0.4)$ with radius 0.3, the maximum of $E_1^{\rm comp}$ is $3.33$, corresponding to a relative error of $11\%$. Within the slab region, the maximum of $E_1^{\rm comp}$ is $2.883$ (15.32\% error), while the maximum of $E_2^{\rm comp}$ in the vertical slab is $2.918$ (16.67\% error), and in the horizontal slab is also $2.918$ (8.43\% error). For $E_3^{\rm comp}$, the maximum value inside the disk is $1.8298$, resulting in a relative error of $8.51\%$. Given the ill-posed nature of the inverse problem and the presence of $10\%$ noise in the data, these reconstruction errors are within a reasonable and acceptable range.

Overall, despite the presence of 10\% multiplicative noise, the proposed method demonstrates strong robustness and fidelity in reconstructing complex spatial structures involving both isotropic and anisotropic features.

\section{Concluding Remarks}\label{sec7}

In this work, we have developed and analyzed a time dimension reduction framework for recovering the initial electric field $\mathbf{E}_0$ in time-domain Maxwell's equations. The proposed method leverages a projection onto a Legendre polynomial-exponential basis in time, effectively transforming the original $(3+1)$-dimensional inverse problem into a sequence of purely spatial problems in three dimensions. This transformation provides significant computational and analytical advantages, particularly by allowing us to relax the requirement of knowing the initial velocity $\partial_t \mathbf{E}(\mathbf{x}, 0)$, which is commonly unavailable in practical settings.

To address the under-determined nature of the problem, we introduced a minimum norm formulation that selects among all admissible solutions the one with smallest joint $L^2((0, T); H^2(\Omega)^3)$ and $H^2((0, T); L^2(\Omega)^3)$ norm. This approach is inspired by the Moore-Penrose-type regularization framework, ensuring both stability and physical plausibility of the reconstruction. We proved a convergence theorem guaranteeing the accuracy of the proposed method under noise-contaminated boundary data.

Our numerical experiments in three dimensions validate the theoretical results and demonstrate the method's robustness and efficiency. In particular, the reconstructions of $\mathbf{E}_0$ achieved reasonable accuracy even in the presence of 10\% noise, with relative errors ranging between 8\% and 16\% across several geometric regions of interest. These results affirm the potential of the time reduction approach as a reliable and computationally feasible tool for solving ill-posed inverse problems in time-dependent Maxwell systems.

In addition to the primary contributions, the proposed method offers several advantages that enhance its practical applicability. First, the time-dimensional reduction approach simplifies the original $(3+1)$-dimensional Maxwell system to a sequence of 3D problems, effectively eliminating the need to specify the initial velocity $\partial_t \mathbf{E}(\mathbf{x}, 0)$. This stands in contrast to traditional methods, which typically require $\partial_t \mathbf{E}(\mathbf{x}, 0)$ as part of the initial data---information that is often inaccessible in real-world applications. Our method bypasses this requirement, thereby broadening its usability in practical scenarios.
 Second, the use of a minimum-norm framework addresses the non-uniqueness that naturally arises in our under-determined problem setting, offering a well-posed strategy to select a physically meaningful solution. Third, the convergence theorem established in this work rigorously confirms the stability of the proposed reconstruction algorithm, even under noisy measurements. Finally, the effectiveness of our approach is demonstrated by a 3D numerical example, where high-fidelity reconstruction is achieved with acceptable error levels despite the ill-posed nature of the inverse problem and the presence of 10\% noise. Together, these contributions highlight the robustness and versatility of our method for solving time-domain inverse problems in electromagnetics.


 \section*{Acknowledgement}
 The work of Dang Duc Trong was supported by Vietnam National University (VNU-HCM) under grant number
 T2024-18-01.
The work of Loc Nguyen was supported by the National Science Foundation grant DMS-2208159.


\end{document}